\documentclass[12pt]{extarticle}
\usepackage{amsmath, amsthm, amssymb, color}
\usepackage{graphicx}
\usepackage{caption}
\usepackage{subcaption}
\usepackage{mathtools}
\usepackage{enumerate}
\usepackage{stackengine}
\usepackage{multirow}
\usepackage{adjustbox}
\usepackage{verbatim}
\usepackage{upgreek, adjustbox}
\usepackage{chemfig, chemnum}
\usepackage{tikz}
\usetikzlibrary{decorations.markings}

\usepackage{tikzscale}
\usetikzlibrary{patterns}
\usepackage[lined,commentsnumbered,ruled]{algorithm2e}
\usepackage{caption}
\usepackage{subcaption}
\usepackage{float}
\oddsidemargin \evensidemargin
\usetikzlibrary{decorations.pathreplacing}
\usepackage{array}
    \newcolumntype{P}[1]{>{\centering\arraybackslash}p{#1}}
    \newcolumntype{M}[1]{>{\centering\arraybackslash}m{#1}}

\usepackage[colorlinks=true, allcolors=teal,linkcolor=teal, citecolor=teal,backref=page]{hyperref}
\usepackage{cleveref}
\usepackage{preamble}

\title{{\bf Tropical KP Theory on Banana Curves}}
\author{Simonetta Abenda, T\"urk\"u \"Ozl\"um \c{C}elik, Claudia Fevola, Yelena Mandelshtam}
\date{}

\begin{document}

\maketitle

\begin{abstract}

The Kadomtsev–Petviashvili (KP) equation is the cornerstone of integrable systems, whose solutions reflect deep connections in algebraic geometry. Banana curves are  reducible rational curves obtained as a degeneration of hyperelliptic curves. In this work, we relate the family of KP multi-solitons arising from banana curves together with non-special divisors of fixed degree to the combinatorics of the tropical theta divisor of the curve. We describe the Voronoi and Delaunay polytopes and show that the latter are combinatorially equivalent to uniform matroid polytopes. As a consequence, the combinatorics of the tropical theta divisor canonically encodes the matroid and Grassmannian structures underlying the associated KP multi-soliton solutions.
We define the Hirota variety of a banana graph, which parametrizes all tau functions arising from such a graph. Starting from the matroid arising from Delaunay polytopes and the periods in the tropical limit, we construct an explicit parametrization of this variety which realizes the tau function as a multi-soliton. Our framework specializes naturally to real and positive settings.
 
\end{abstract}

\section{Introduction}

The Kadomtsev–Petviashvili (KP) equation is a nonlinear partial differential equation
\begin{equation}\label{eq:KPequation}
(-4u_t + 6uu_x + u_{xxx})_x \,=\, -3u_{yy},
\end{equation}
for an unknown function $u(x,y,t)$ of two spatial variables $x,y$ and one time variable $t$, where subscripts denote partial derivatives. When the variables $x,y,t$ are real, the form of~\eqref{eq:KPequation} is the KP II equation. In the rest of the paper, for simplicity, we will refer to~\eqref{eq:KPequation} as the KP equation both in the real and complex settings.  
This prominent integrable equation models the propagation of long, weakly two-dimensional shallow water waves with slow variation along the transverse $y$-direction. The KP equation is among the most fundamental examples in the theory of integrable systems, encompassing the Korteweg de Vries, Boussinesq, and Toda equations as reductions or particular limits. It serves as an example of an integrable hierarchy whose solutions possess dual geometric realizations through algebraic curves on one side and Grassmannians on the other. Understanding how these two realizations are related has been a central problem in the modern theory of integrable systems.

Two distinguished families of solutions are the finite-gap and the multi-soliton solutions.  
Finite-gap solutions are parametrized by algebraic curves, namely the spectral curves of the Lax operator, and are expressed in terms of Riemann theta functions.  
KP~solitons admit a representation via Wronskian determinants and are parameterized by points in finite-dimensional Grassmannians. Real and regular solitons correspond to totally nonnegative Grassmannians ${\rm Gr}_{\geq 0}(k,n)$ \cite{CK, KW2013,KW2014}, while real and regular finite-gap solutions are parameterized by divisors on M-curves fulfilling  Dubrovin--Natanzon conditions \cite{DubNat}.

The relation between these two classes has been deeply studied in recent years. Abenda and Grinevich \cite{Abenda2018,Abenda2019,Abenda2022, AbendaGrinevich2025DN} established that every real regular KP soliton arises as the limit of a real regular finite-gap solution, providing an explicit characterization of the associated spectral data.  
In particular, in \cite{Abenda2022}, they showed that each plabic graph representing a positroid cell is dual to a rational degeneration of an M-curve, and that each soliton datum associated with the cell corresponds uniquely to a divisor on the curve satisfying the Dubrovin–Natanzon reality and regularity conditions.

An alternative viewpoint on this correspondence comes from tropical geometry. In the tropical limit, the Riemann theta function becomes a finite sum of exponentials supported on the vertices of a polytope, called Delaunay polytope, as introduced in~\cite{AgoCelStrStu,AgoFevManStu}. In this framework, soliton solutions appear as rational degenerations of finite-gap solutions, with their combinatorial structure encoded by tropical curves and their Jacobians. The Hirota variety parametrizes the $\tau$-functions arising from a degenerate Riemann theta function that satisfy the Hirota bilinear relation \cite{AgoFevManStu, FevMan}.  Degeneration of Riemann theta functions have been further explored also in e.g.,~\cite{ichikawa2023periods,ICHIKAWA2024114748}. 

In this paper, we explore this correspondence for tropical degenerations associated with banana graphs, i.e.,  metric graphs with two vertices connected by $g+1$ edges, which mainly represent the tropical limits of hyperelliptic curves of genus $g$. We describe their tropical Jacobians and Voronoi--Delaunay decomposition, and establish their connections with matroid base polytopes, such as hypersimplices. 
A key aspect of our construction is the realization of the tropical Jacobian of the banana graph in a higher-dimensional space, where it appears as a projection of the $(g+1)$-dimensional hypercube. This representation plays a central role in relating the combinatorial structure of the tropical Jacobian to Grassmannians and soliton~$\tau$-functions. Our results thus provide a concrete geometric realization of the correspondence between tropical degenerations of algebraic curves and KP solitons, making explicit how tropical Jacobians and combinatorial polytopes capture the algebraic structure of integrable hierarchies. 

The present paper is the first step in a broader project aimed at understanding KP solitons arising from arbitrary metric graphs. We treat the case of banana graphs in full detail not only because it is already rich enough to exhibit several key phenomena, but also because it provides an informative model illustrating the methods that will be used in greater generality. This paper should be viewed as a detailed case study that foreshadows a much more general theory. Its full development is part of ongoing joint work by the authors.

The main results of this paper describe the combinatorial and geometric structure underlying KP multi-solitons arising from tropical degenerations associated with banana graphs. We give an explicit description of the Voronoi polytope of the tropical Jacobian of a banana graph, including its set of vertices and $f$-vector (Theorem~\ref{thm:vertices_voronoi_Rg}), and show that the associated Delaunay polytopes are combinatorially equivalent to hypersimplices, hence to matroid base polytopes of uniform matroids (Theorem~\ref{thm:BTD} and ~\Cref{cor:Delaunay_in_Rg}). We further identify the face poset of the Voronoi polytope with the poset of strongly connected orientations of the graph (see~\cite[Theorem~1]{amini2010lattice}), and provide an explicit description of this correspondence in the case of banana graphs (\Cref{thm:bijection}). As a consequence, each Voronoi vertex, together with a choice of graph vertex, canonically determines a matroid whose bases are in bijection with the corresponding Delaunay set, thereby linking the tropical Jacobian to the Grassmannian framework underlying KP solitons (Theorem~\ref{thm:delsetismatroid}).

Building on this combinatorial description, we introduce the Hirota variety of a graph, which parametrizes $\tau$-functions arising from tropical limits of Riemann theta functions satisfying the Hirota bilinear relations. For banana graphs, we describe the main component of this variety explicitly: Theorem~\ref{theo:Hirota_parametrization} gives a parametrization in terms of tropical data that realizes all such $\tau$-functions as KP multi-solitons. In this framework, the matroid underlying a KP soliton coincides with the matroid determined by the associated Voronoi vertex and graph orientation (\Cref{prop:matroid_A}). We further show that this parametrization agrees with previously known descriptions \cite{Abenda2017,nakdegeneration} of KP multi-solitons arising from rational degenerations of hyperelliptic curves (Theorems~\ref{theo:Krichever1_Hirota} and \ref{the:paramHirotagenericD}). Finally we define the real and positive Hirota variety, relating it to totally nonnegative Grassmannians and MM-curves as in \cite{Abenda2017,kummer2024maximal}. 
Equivalently, it parametrizes real $\tau$–functions of a graph satisfying Hirota bilinear relation with positivity constraints on their coefficients. This makes the positive Hirota variety a natural object from the viewpoint of positive geometry \cite{Lam2024, Lam2025, ranestad2025positive}, providing a tropical realization of the structures underlying real and regular KP solitons. 

The paper is organized as follows. In~\Cref{sec2}, we review the background supporting the connection between KP multi-solitons and tropical degenerations of algebraic curves, including tropical curves, their Jacobians, and the relevant graph homology theory. ~\Cref{sec3} is devoted to the Voronoi–Delaunay combinatorics of banana graphs and its interpretation in terms of strongly connected orientations and matroids. In \Cref{sec:theta_functions_limit}, we recall KP multi-solitons and compute periods of Abelian differentials in the tropical limit, which then are used to express the solitons. ~\Cref{sec:Hirota_of_graph} introduces the Hirota variety of a banana graph and gives explicit parametrizations of its main component, including real and positive loci. Finally, in~\Cref{sec:Hirota_of_hypers}, we define the Hirota variety associated with the Delaunay polytopes of the banana graph and study its defining polynomial equations.

\section{Geometric and combinatorial foundations}\label{sec2}

In this section, we assemble the necessary background from tropical geometry, combinatorial and homological models of degenerating curves. We review the formalisms that will later support our connection between KP multi-solitons and tropical degenerations of algebraic curves. We begin with the tropical perspective on moduli spaces of curves and abelian varieties, where tropical curves are modeled by weighted metric graphs, and their Jacobians are realized as real tori built from the combinatorial and metric structure of the graph. 
 We then discuss how this picture extends naturally to a homological setting, where graphs are treated as chain complexes, and Jacobians arise from the space of cycles modulo boundaries. In this setting, degenerations and compactifications are governed by the associated lattice and fan structures, as studied in tropical geometry. Our main references are \cite{brannetti2011tropical,chan2012combinatorics, chua2019schottky, oda1979compactifications}.

\subsection{Tropical curves and their Jacobians}\label{sec:tropicalJacobians}
A \textit{metric graph} is a pair $(\Gamma,l)$, where $\Gamma=(V,E)$ is a finite connected graph consisting of a finite set of vertices $V=\{v_1,\dots,v_r\}$ and edges $E=\{e_1,\dots,e_m\}$, loops and parallel edges allowed, and $l$ is a function $l:E\rightarrow \RR_{>0}$ on the edges of $\Gamma$ that records lengths of the edges of $G$. The \emph{genus} of a graph $\Gamma$ is the rank of its first homology group $g \coloneqq g(\Gamma) = |E|-|V|+1.$ A \textit{tropical curve} is usually defined as a triple given by the metric graph and an $\NN$-valued function which assigns non-negative integer weights to the vertices of $\Gamma$ with the property that every weight-zero vertex has degree at least $3$ \cite{chan2012combinatorics}. In this article, we are interested in tropical curves that arise as rational degenerations of smooth complex algebraic curves, therefore, we consider tropical curves having weight zero at each vertex. In such cases, the arithmetic genus of the curve coincides with the genus of its metric graph.

Given a tropical curve $(\Gamma,l)$, one can determine its \textit{tropical Riemann matrix}. Consider the first integral homology group $H_1(\Gamma,\ZZ)\simeq \ZZ^g$ of the graph $\Gamma$. Let $B$ denote the $g\times m$ matrix whose columns record the coefficients of each edge in the basis vectors, and let $\Delta$ denote the $m\times m$ diagonal matrix whose entries record the edge lengths. The Riemann matrix of $\Gamma$ is the positive definite $g\times g$ matrix 
\begin{equation}\label{eq:tropical_Riemann_matrix}
    Q \, =\, B  \cdot  \Delta \cdot B^T \,=\, l(e_1)\bb_1\bb_1^T+l(e_2)\bb_2\bb_2^T+\cdots+l(e_m)\bb_m \bb_m^T,
\end{equation}
where $\bb_1,\bb_2,\dots,\bb_m$ are the column vectors of the matrix $B$. In particular, it defines the norm on $\mathbb{R}^g$ as $Q(x)=||B^T \cdot x||_l$ where $||y||_l=\sum^m_{i=1}l(e_i)y_i^2$. The \emph{tropical Jacobian} of $\Gamma$ is the torus $\RR^g/\ZZ^g$ together with the matrix $Q$. A basis of $H_1(\Gamma,\ZZ)$ can be determined by fixing an arbitrary orientation of the edges of $\Gamma$ and a spanning tree, this procedure is described in detail in \cite[Section~4]{bolognese2017curves} together with an example for a graph of genus $4$. A different basis choice gives another matrix related by a $\text{GL}_g(\ZZ)$-action by conjugation.
Given the expression in \eqref{eq:tropical_Riemann_matrix}, the cone of all matrices that are Riemann matrices of the graph $\Gamma$, allowing all edge lengths to vary, is the rational open polyhedral cone 
\[\sigma_{\Gamma}\, = \, \RR_{>0}\{\bb_1\bb_1^T,\bb_2\bb_2^T,\dots,\bb_m\bb_m^T\}.\]
Let $\tilde{S}^g_{\geq 0}$ be the set of $g\times g$ symmetric positive semidefinite matrices with rational nullspace, namely with kernels having bases defined over $\QQ$. As for tropical Riemann matrices, the group $\text{GL}_g(\ZZ)$ acts on $\tilde{S}^g_{\geq 0}$ by conjugation. Given a fixed tropical Riemann matrix $Q\in \tilde{S}^g_{\geq 0}$, the values of its quadratic form $\cc \mapsto \cc^TQ\cc$ in $\RR$ with $\cc \in \ZZ^g$ define a regular polyhedral subdivision of $\RR^g$ with vertices at $\ZZ^g$. This is known as the \textit{Delaunay subdivision} of $Q$ and denoted $\text{Del}(Q)$. The dual subdivision to $\text{Del}(Q)$ is known as the \textit{Voronoi decomposition} of $\RR^g$. The \emph{tropical theta divisor} is given by the faces of codimension one in the Voronoi decomposition. The cells of the Voronoi decomposition of $Q$ are the lattice translates of the \textit{Voronoi polytope}
\begin{equation}\label{eq:Voronoi_polytope}
 V_{Q}\,\coloneqq \,\{\aa\in\RR^g \,\, :\,\, \aa^TQ \aa \leq (\aa-\cc)^TQ (\aa-\cc) \,\,\text{for all } \cc\in \ZZ^g\}.
\end{equation}
Note that $V_Q$ consists of all points in $\RR^g$ such that the origin is the closest lattice point, in the norm induced by $Q$. We denote by $\mathcal{V}_Q$ the set of its vertices. The $f$-vectors of the Voronoi polytopes of graphs in genus $4$ are given in~\cite[Table~1]{chua2019schottky}. The following result was proven in \cite[Proposition 3.1]{chua2019schottky} building on \cite{vallentin2003sphere}. 
See \cite[Section 4.4]{brannetti2011tropical} for more details.

\begin{proposition}
    Let $(\Gamma,l)$ be a metric graph, homology basis represented by the matrix $B$, and Riemann matrix $Q = B\Delta B^T.$ The Voronoi polytope~\eqref{eq:Voronoi_polytope} is affinely isomorphic to the zonotope $\sum_{i=1}^m[-\bb_i,\bb_i]$.
\end{proposition}

A zonotope is a polytope that can be realized as a Minkowski sum of segments, or equivalently, that can be obtained as an aﬃne projection of an hypercube, see \cite[Chapter~7]{ziegler2012lectures}.

We are interested in the combinatorial structure of polytopes that arise in the Delaunay subdivision of $\text{Q}$. Fixing $\aa$ to be a vertex of the Voronoi polytope $V_Q$ in~\eqref{eq:Voronoi_polytope}, we will refer to such vertices as \emph{Voronoi vertices}. The \textit{Delaunay set} associated with $\aa$ and $Q$ is given as
\begin{equation}\label{eq:Delaunay_set}
    \mathcal{D}_{\aa,Q}\, = \, \{\cc \in \ZZ^g \, : \, \aa^TQ\aa = (\aa-\cc)^TQ(\aa-\cc)\}.
\end{equation}
We denote $D_{\aa,Q}$ the Delaunay polytope obtained as the convex hull of the set of integer vectors in $\mathcal{D}_{\aa,Q}$. Note that different choices of Voronoi vertices can lead to different types of Delaunay polytopes, a list of such polytopes in genus~$3$ is given in~\cite[Section~4]{AgoCelStrStu}. 

Note that the natural
action of $\text{GL}_g(\ZZ)$ on $\Sg$ corresponds to an action on Voronoi and Delaunay polytopes, by left multiplication by a matrix $G^{-1}$. Given a matrix $\tilde{Q} = G^TQG$ where $Q\in \Sg$ and $G\in\text{GL}_g(\ZZ)$, then $\mathcal{V}_{\tilde{Q}}$ contains the vectors $G^{-1}\aa$ for $\aa\in\mathcal{V}_Q$. Analogously for the Delaunay sets, we have $\mathcal{D}_{G^{-1}\aa,\tilde{Q}}=\{G^{-1}\cc\, : \, \cc\in \mathcal{D}_{\aa,Q}\}$.

For simplicity, in the rest of the paper we assume the $m\times m$ matrix $\Delta$ to be the identity matrix. This assumption is not restrictive as we are interested only in the combinatorics of the Voronoi and Delaunay subdivisions, and this is invariant with respect to the choices of edge lengths \cite{brannetti2011tropical, chan2012combinatorics}.

\begin{proposition}\label{prop:preequiv}
Let $V_Q$ be the Voronoi polytope associated to a graph $\Gamma$, and let $\aa$ be a vertex of the $V_Q$. Then, for any $\cc_0\in \DaQ$, the following hold: 
\begin{enumerate}
    \item The point $\aa' = \aa-\cc_0$ is also a vertex of $V_Q$.
    \item The associated Delaunay set satisfies 
    ${\cal D}_{\aa',Q} = {\cal D}_{\aa,Q} -\cc_0 $.
    \item The point $-\cc_0$ belongs to $ \mathcal{D}_{\aa',Q}$.
    \item Moreover, if $\cc'_0\in \mathcal{D}_{\aa',Q}$ and $\aa'' = \aa' -\cc'_0$
    , then there exists $\cc\in \DaQ$ such that $\aa'' = \aa-\cc$.
\end{enumerate}
\end{proposition}

\begin{proof}
By construction, $\aa'= \aa-\cc_0$ is a vertex of the Voronoi decomposition associated with $V_Q$, proving Item $1$. We now show Item $2$. By definition, a point $\cc' \in \mathcal{D}_{\aa',Q}$ if and only if it satisfies the equality in the definition in \eqref{eq:Delaunay_set}. Substituting $\aa' = \aa-\cc_0,$ we obtain
\begin{equation*}
    0 = (\aa-\cc_0)^T Q (\aa-\cc_0) - (\aa -\cc_0 -\cc')^T Q (\aa -\cc_0 -\cc')
     = \aa^T Q \aa -(\aa -(\cc_0 +\cc'))^T Q (\aa -(\cc_0 +\cc')).
\end{equation*}

Vice versa, for any $\cc \in \DaQ$, we have $\cc-\cc_0 \in \mathcal{D}_{\aa',Q}$. In fact,
\[
(\aa')^T Q\aa' - (\aa'-\cc +\cc_0)^T Q(\aa'-\cc+\cc_0) = \aa^T Q \aa -(\aa-\cc)^T Q (\aa- \cc) =0
\]
Item $3$ is the special case $\cc'=\mathbf{0}$, so that $\cc=\cc_0$, that is $\aa'$ is a vertex of $V_Q$ and $-\cc_0 \in \mathcal{D}_{\aa',Q}$. Finally, Item $4$ is an immediate consequence of Item $2$. We have $\aa''=(\aa-\cc_0)-(\cc-\cc'_0)=\aa-\cc$, for some $\cc\in\DaQ$. This concludes the proof.
\end{proof}

Motivated by \Cref{prop:preequiv}, we introduce an equivalence relation on the set of vertices of the Voronoi polytope. In fact, \Cref{prop:preequiv} shows that translating a Voronoi vertex $\aa$ by any element of $\DaQ$ produces another Voronoi vertex, and that the corresponding Delaunay polytopes transform compatibly. This suggests grouping those vertices related by such admissible translations
\begin{definition}\label{eq:equivalence_Voronoi_vertices}
    Let $\aa, \aa'$ be vertices of a Voronoi polytope associated with a graph $\Gamma$, we say the $\aa$ and $\aa'$ are \emph{equivalent (under Delaunay-translation)}, and write
\begin{equation*}
\aa\, \sim\, \aa', \textit{ if there exists } \cc_0 \,\in \,\DaQ \textit{ such that } \aa'\,=\,\aa-\cc_0. 
\end{equation*}
We denote by $[\aa]$ the equivalence class of a vertex $\aa\in \mathcal{V}_Q$ with respect to the Delaunay-translation equivalence.
\end{definition}

The following characterization holds for any class of Voronoi and Delaunay polytopes associated with a graph $\Gamma$ as defined in~\Cref{sec2} and  ensures that the Delaunay-translation equivalence is a well-defined equivalence relation on $\mathcal{V}_Q.$ 

\begin{proposition}\label{prop:equivvertex_general}
 Let $V_Q$ be the Voronoi polytope associated to a graph $\Gamma$, and let $\aa,\aa' \in {\cal V}_Q$. Then, $\aa\sim\aa'$ if and only if $\mathcal{D}_{\aa',Q} = \mathcal{D}_{\aa,Q} -\cc_0 $, where $\cc_0\in \mathcal{D}_{a,Q}$ satisfies $\aa'=\aa-\cc_0.$  
\end{proposition}
\begin{proof}
The necessary condition follows from \Cref{prop:preequiv} and \Cref{eq:equivalence_Voronoi_vertices}. For the sufficient condition, assume that $\mathcal{D}_{\aa, Q} = \mathcal{D}_{\aa', Q} - \mathbf{\cc_0}$ for some point $\mathbf{\cc_0} \in \mathcal{D}_{\aa, Q}$. As $\aa$ is the unique point equidistant with respect to the $Q$-norm to all points of the set $\mathcal{D}_{\aa, Q}$, the point $\aa-\cc_0$ must be equidistant to all points of $\mathcal{D}_{\aa, Q} -\cc_0 = \mathcal{D}_{\aa', Q}$. On the other hand, $\aa$ is the unique such point (it is the circumcenter of $\mathcal{D}_{\aa, Q}$ with respect to the $Q$-norm), then it must be the case that $\aa = \aa'+ \cc_0$.
\end{proof}

 Combining \Cref{prop:preequiv} with \Cref{prop:equivvertex_general} we get the following. 
\begin{corollary}\label{cor:conv_ak}
    Let $\aa$ be a vertex of the Voronoi polytope $V_Q$. Then the convex hull of the points in $[\aa]$ verifies $\text{Conv}([\aa]) = \aa-D_{a,Q}$. In particular, $|[\aa] | = |\DaQ|$. 
\end{corollary}

\begin{example}[Genus $2$]\label{ex:sec1_genus2_banana}
Let $\Gamma=(V,E)$ be the banana graph of genus $2$, see Figure~\ref{fig:bananagen3}, or \cite[Figure~2 (right)]{AgoFevManStu} for the case of genus $2$. With the orientation given in Figure~\ref{fig:bananagen3}, a cycle basis is given by $e_1-e_2, e_1-e_3$. Thus, a matrix representing $H_1(\Gamma,\ZZ)$ and the tropical Riemann matrix are  
    \[
    B_2 \,=\, \begin{bmatrix}
        1 & -1 & 0\\
        1 & 0 & -1
    \end{bmatrix} \quad \text{and}\quad Q_2 \,=\, \begin{bmatrix}
        2 & 1\\
        1 & 2
    \end{bmatrix}.
    \]
    The Voronoi polytope $V_Q$ is the hexagon with vertices
\begin{equation*}
\begin{aligned}
\aa_1 &= \left(-\tfrac{1}{3}, \tfrac{2}{3}\right), &
\aa_2 &= \left(-\tfrac{1}{3}, -\tfrac{1}{3}\right), &
\aa_3 &= \left(\tfrac{2}{3}, -\tfrac{1}{3}\right), \\[2pt]
\aa_4 &= \left(\tfrac{1}{3}, \tfrac{1}{3}\right),  &
\aa_5 &= \left(-\tfrac{2}{3}, \tfrac{1}{3}\right), &
\aa_6 &= \left(\tfrac{1}{3}, -\tfrac{2}{3}\right).
\end{aligned}
\end{equation*}
     This is illustrated in ~\Cref{fig:hexagon} together with the associated Delaunay polytopes. The vertices of the hexagon form two Delaunay-translation equivalence classes, namely $\{\aa_1,\aa_2,\aa_3\}$ (in blue)  and $\{\aa_4,\aa_5,\aa_6\}$ (in green).    
    \begin{figure}[htbp]
        \centering
        \captionsetup{width=.8\linewidth, font=small}
        \includegraphics{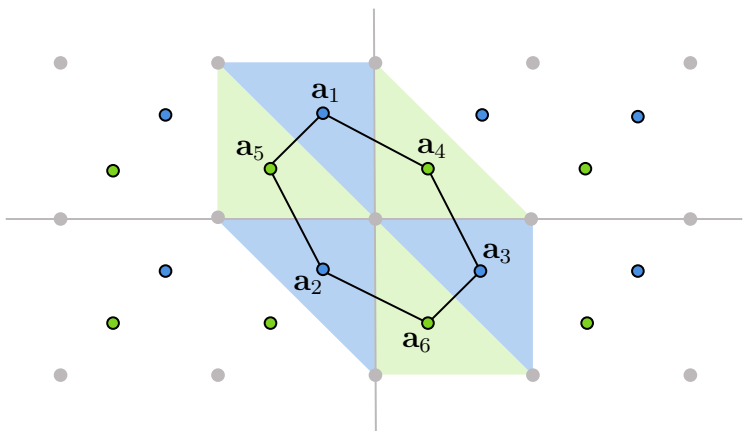}
        \caption{The Voronoi (hexagonal) and Delaunay (triangular) polytopes associated with the banana graph of genus 2. The colors highlights the two equivalence classes of vertices in ${\cal V}_Q$ under Delaunay-translation equivalence. The gray points are the lattice $\ZZ^2$, while the blue (resp.~green) points are the integer translates of the blue (resp. green) vertices of the Voronoi polytope.} 
        \label{fig:hexagon}
\end{figure}
\end{example}

\subsection{Chain complexes and degenerations of curves}\label{sec:OSbackground}

We now review the main notions and results concerning Voronoi and Delaunay decompositions that will be employed in the proofs of the upcoming section. Let $\Gamma=(V,E)$ be a finite connected graph as in Section~\ref{sec2}. The results can also be adapted to the case of disconnected graphs \cite{oda1979compactifications}, however, for our purposes, we deal with connected graphs.
We allow loops and multiple edges. Given a graph $\Gamma$, we assign and fix an \textit{orientation} of $\Gamma$, that is, a map $\iota: E\to V\times V$ such that $\iota(e)=(e_{-},e_+)$ where $e_{-},e_+$ denote respectively the source and target vertices. We denote $\iota(-e)=(e_+,e_-).$ 
Moreover, when $e$ is a loop, $\iota(e)$ denotes the loop oriented clockwise, while $\iota(-e)$ denotes the loop oriented counterclockwise. The results in this section are independent of the chosen orientation. 

Let $C_{\bullet}(\Gamma,S)$ denote the chain complex 
where the spaces of $0$-chains and $1$-chains with coefficients in a finite abelian group $S$ (we will use the groups $S=\ZZ,\RR$) are 
$$C_0(\Gamma,S) \,=\, \bigoplus_{v_i\in V}S\cdot v_i,\qquad\quad \qquad\quad C_1(\Gamma,S) \,=\, \bigoplus_{e_j\in E}S \cdot e_j,$$ where $\{v_i\}_{v_i\in V}$ and $\{e_j\}_{e_j\in E}$ canonically define $S$-bases. The above spaces are endowed with the canonical pairings $[\,\cdot\,\,,\,\cdot \,]$ and $(\,\cdot\,\,,\,\cdot \,)$ on $C_0(\Gamma,S)$ and $C_1(\Gamma,S)$, respectively, given by
\[[v_i,v_{i'}]\,=\, \delta_{ii'}, \qquad \text{and}\qquad (e_j,e_{j'})\,=\, \delta_{jj'},\]
where $\delta_{kk'}$ is the Kronecker delta. The boundary map  $\partial: C_1(\Gamma,S)\rightarrow C_0(\Gamma,S)$ is defined by 
\[\partial e_j \, =\, \begin{cases}
    0 & \text{if $e_j$ is a loop}\\
    v_{i'}-v_i & \text{if $\iota(e_j)=(v_i,v_i')$ 
    }.
\end{cases}\]
 The canonical pairings allow to identify the cochain complex $C^{\bullet}(\Gamma,S)$ with the chain complex $C_{\bullet}(\Gamma,S)$. The coboundary map becomes $\delta: C_0(\Gamma,S)\rightarrow C_1(\Gamma,S)$ given by
\[\delta v_i \,=\, \sum_{e_j \in E} [v_i,\partial e_j]e_j.\]
One can easily check that $\partial$ and $\delta$ are adjoint to each other, i.e., we have $(\delta x,y) = [x,\partial y]$ for $x \in C_0(\Gamma,S)$ and $y \in C_1(\Gamma,S)$.
Note that, when $S=\ZZ$, the kernel of $\partial$ is the first homology group of $\Gamma$ with coefficients in $\ZZ$, denoted $H_1(\Gamma,\ZZ)$.
By the adjointness of $\partial$ and $\delta$, we get the following decompositions orthogonal with respect to the pairings: 
\begin{equation}\label{eq:direct_sums}
    C_0(\Gamma,S) \,=\, H^0(\Gamma,S) \oplus \partial C_1(\Gamma,S), \quad C_1(\Gamma,S) \,=\, H_1(\Gamma,S) \oplus \delta C_0(\Gamma,S),
\end{equation}
where $H^0(\Gamma, S)$ denotes the kernel of $\delta$.
If we write the corresponding Laplacians 
as 
\begin{equation}\label{eq:Laplacians}
    \Delta_0 \,=\, \partial \delta : C_0(\Gamma,S) \rightarrow C_0(\Gamma,S) \quad\text{and}\quad \Delta_1 \,=\, \delta \partial: C_1(\Gamma,S) \rightarrow C_1(\Gamma,S) 
\end{equation}
then the decompositions in \eqref{eq:direct_sums} can be written as
\begin{equation}
     C_0(\Gamma,S) \,=\, \ker (\Delta_0) \oplus \partial C_1(\Gamma,S), \quad C_1(\Gamma,S) \,=\, \ker (\Delta_1)  \oplus \delta C_0(\Gamma,S). 
\end{equation}
Moreover $\partial$ and $\delta$ respectively induce the isomorphisms $\delta C_0(\Gamma,S)  \stackrel{\partial}{\simeq} \partial C_1(\Gamma,S)$ and $\partial C_1(\Gamma,S)  \stackrel{\delta}{\simeq} \delta C_0(\Gamma,S).$
Let $\pi':C_1(\Gamma,S)\rightarrow H_1(\Gamma,S)$ and $\pi'':C_1(\Gamma,S)\rightarrow \delta C_0(\Gamma,S)$ denote the orthogonal projections induced by the decomposition in \eqref{eq:direct_sums} on the right, and $\rho: C_1(\Gamma,S)\rightarrow H^1(\Gamma,S)$ the canonical surjection. Then we have 
\begin{equation}\label{eq:rho_pi_prime}
\partial \circ \pi'', \, = \, \partial  \qquad\qquad \rho \circ \pi' \, = \, \rho.
\end{equation}
For fixed orientation and a homology basis for the graph $\Gamma$, the surjection $\rho$ is explicitly represented by the matrix $B$ from \eqref{eq:tropical_Riemann_matrix}. 

We identify $C_1(\Gamma,\RR)$ with $\RR^m$, endowed with its standard Euclidean metric. Under this identification, $\Lambda:=C_1(\Gamma,\ZZ)\subset C_1(\Gamma,\RR)$ corresponds to the lattice $\ZZ^m$, and similarly $\Lambda\cap H_1(\Gamma,\RR) = H_1(\Gamma,\ZZ) \simeq \ZZ^g$, which also coincides with the image of the vertices of the Delaunay decomposition in $\RR^g$. Indeed, each vector $B^T\cc$ represents a cycle in $H_1(\Gamma,\ZZ)$. The image of the Voronoi decomposition in $C_1(\Gamma,\RR)$ is given by the integer translates
$B^T(V_Q+{\bf c})$, with ${\bf c}\in \ZZ^g$. One checks that $B^T(V_Q)$ is itself a polytope, obtained as the convex hull of the image under $B^T$ of the vertices of $V_Q$.
By \eqref{eq:rho_pi_prime}, we have $B^T(V_Q)\subset H_1(\Gamma,\RR)$. Next, the fundamental cell of 
$\Lambda^\vee$ 
is the hypercube $C^m$, with 
\begin{equation}\label{eq:cube}
C^m \, \coloneqq \, \left[-\frac{1}{2},\frac{1}{2}\right],
\end{equation}
the hypercube in $\RR^m$ with vertices the $2^m$ vectors in $\{-\tfrac{1}{2},\tfrac{1}{2}\}^m$.
In fact $\Lambda^\vee=\Lambda + (\frac{1}{2},\dots, \frac{1}{2})$. Since a Voronoi polytope is in particular a zonotope, it can be realized as the projection of a hypercube. In the ambient space $\RR^m$, this projection becomes explicit: we have $B^T(V_Q)=\pi'(C^m)$ and the Voronoi subdivision is identified with $\pi'(\Lambda^{\vee})$. For more details about Voronoi and Delaunay decompositions for graphs we refer to \cite[Section~5]{oda1979compactifications}. 

Let $M\in \ZZ^{n\times m}$ be an integer matrix, and let $P\subset \RR^m$ be a polytope with vertices in $\RR^m$. We denote by $M(P)$ the polytope defined as the convex hull of the points $M\cdot \aa$, where $\aa$ ranges over the vertices of $P$. The advantage of working with $1$-chains lies in the fact that, in this space, the right-hand expression in \eqref{eq:rho_pi_prime} provides a realization of the Voronoi polytope as the projection of a hypercube, namely, we have $B^T(V_Q)=\pi'(C^m)$.
   
\begin{example}[Genus $2$]
Let $\Gamma$ be as in Example~\ref{ex:sec1_genus2_banana}. The chain groups are $C_0(\Gamma,\ZZ)=\ZZ \cdot v_1 \oplus \ZZ \cdot v_2$ and $C_0(\Gamma,\ZZ)=\ZZ \cdot e_1 \oplus \ZZ \cdot e_2 \oplus  \ZZ \cdot e_3$, with boundary maps and Laplacians represented by the matrices
\begin{equation*}
\partial\,=\,\begin{bmatrix}
        -1 & -1 & -1\\
        1 & 1 & 1 
    \end{bmatrix},
    \qquad
\delta\, = \, \begin{bmatrix}
        -1 & 1 \\
        -1 & 1 \\
        -1 & 1 \\
    \end{bmatrix},\qquad
\Delta_1\, = \, \begin{bmatrix}
    3 & -3\\
    -3 & 3
\end{bmatrix},\qquad
\Delta_2\, = \, \begin{bmatrix}
    2 & 2 & 2\\
    2 & 2 & 2\\
    2 & 2 & 2
\end{bmatrix}.
\end{equation*}
The space of cycles, which is $H_1(\Gamma,\RR)=\ker(\partial)$, represented by the matrix $B$ in Example~\ref{ex:sec1_genus2_banana} and the canonical projection $\pi':C_1(\Gamma,\RR)\to H_1(\Gamma,\RR)$ is given by 
\begin{equation}\label{eq:pi_prime_example}
\pi'\, = \,
\begin{bmatrix}
\frac{2}{3} & -\frac{1}{3} & -\frac{1}{3} \\[1pt]
-\frac{1}{3} & \frac{2}{3} & -\frac{1}{3}  \\[1pt]
-\frac{1}{3} & -\frac{1}{3} & \frac{2}{3} 
\end{bmatrix}.
\end{equation} 
The Voronoi polytope $V_Q\subset\RR^2$ can be lifted to $\RR^3$ via the map represented by $B^T$. One can check that its vertices coincide with the image of the vertices of the cube $C^3=[-\tfrac{1}{2},\tfrac{1}{2}]^3$ via the projection $\pi'$. This gives a hexagon in $\RR^3$ with $\pi'(\pm(\tfrac{1}{2},\tfrac{1}{2},\tfrac{1}{2}))=(0,0,0)$ lying in the interior of the hexagon. This is illustrated in~\Cref{fig:hexagon_orientations}. 
\end{example}

\section{Voronoi--Delaunay combinatorics of banana graphs}\label{sec3}

We now specialize the general framework to the case of tropical curves whose dual graphs are banana (or dipole) graphs, see \cite{amini2023tropical,haase2012linear}. Unless stated otherwise, this will be the case for the rest of the paper. Their Voronoi–Delaunay geometry is particularly explicit and thus allows us to conduct a careful study of its associated Hirota varieties and KP solitons, which comprises the remainder of this paper.

Let $\Gamma_g =(V,E)$ be the banana metric graph with two vertices $v_1,v_2$ connected by $n= g+1$ edges $e_1,\dots,e_{n}$, representing the branch points in~\eqref{eq:hyperelliptic_curve}. Figure~\ref{fig:bananagen3} illustrates the graph $\Gamma_g$. In what follows, we will omit the subscript $g$, unless necessary.
\begin{figure}
        \centering
        \captionsetup{width=.8\linewidth, font=small}
\begin{tikzpicture}
    \node[circle, draw, fill=black, inner sep=2pt, label=left:$v_1$] (u) at (0,0) {};
    \node[circle, draw, fill=black, inner sep=2pt, label=right:$v_2$] (v) at (3,0) {};
    
    \draw[thick,decoration={markings, mark=at position 0.5 with {\arrow{<}}},
          postaction={decorate}] 
        (u) to[bend left=60] node[above] {$e_1$} (v);
        
    \draw[thick,decoration={markings, mark=at position 0.5 with {\arrow{<}}},
          postaction={decorate}] 
        (u) to[bend left=22] node[midway, above=0pt, yshift=-2pt] {$e_2$} (v);
        
    \draw[thick,decoration={markings, mark=at position 0.5 with {\arrow{<}}},
          postaction={decorate}] 
        (u) to[bend right=22] 
        (v);
        
    \draw[thick,decoration={markings, mark=at position 0.5 with {\arrow{<}}},
          postaction={decorate}] 
        (u) to[bend right=60] node[below] {$e_{n}$} (v);

    \node at (1.5, 0.1) {$\vdots$};
\end{tikzpicture}

        \caption{The banana graph of genus $g$ with orientation $\iota_0$.}
        \label{fig:bananagen3}
    \end{figure}
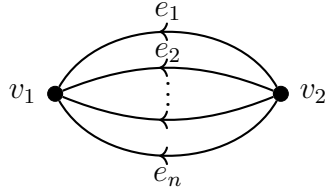
We denote $\iota_0$ the orientation on $\Gamma$ such that $\iota_0(e)=(v_2,v_1)$ for all $e\in E$.
The tropical Riemann matrix of $\Gamma$ is defined in \eqref{eq:tropical_Riemann_matrix}, where we choose the rows of the matrix $B$ to be given by the cycle basis
$H_1(\Gamma,\ZZ)\,=\, \langle e_1-e_2, e_1-e_3, \dots , e_1-e_{g+1}\rangle$. For the choice of $l(e_i)=1,$ for $ i=1,\dots,g+1$, we get that $Q$ has $Q_{ii}=2$ and $Q_{ij}=1$ for all $i,j=1,\dots,g$ and $i\neq j$. We are interested in the combinatorics of the Voronoi and Delaunay polytopes arising from this family of graphs.

\subsection{Voronoi polytopes}

Before stating the structure of the Voronoi polytope explicitly, it is informative to note a key feature of banana graphs: the matrix $Q$ has constant diagonal and off-diagonal entries, so the associated norm is symmetric under permutations of the edges. As a result, we find that the Voronoi polytope is the projection of the $(g+1)$-dimensional cube onto the hyperplane of coordinate sum zero. The following theorem records the resulting family of vertices and their combinatorics.

\begin{theorem}\label{thm:vertices_voronoi_Rg}
   Let $\Gamma$ be a genus-$g$ banana graph. The associated Voronoi polytope $V_{Q}\subset \RR^g$ is a polytope of dimension $g$ 
   with exactly $2(2^{g}-1)$ vertices $\bigcup_{i=1}^g \,[{\bf k}]$ where    
\begin{equation}\label{eq:bananavorvert}
 [{\bf k}] \coloneqq  \Biggl\{        \text{permutations of }\Bigg( \underbrace{ -\frac{k}{g+1},\, \dots,\,  -\frac{k}{g+1}}_{g-k \text{ or } g+1-k \text{ times}}, \underbrace{\frac{g+1-k}{g+1},\,\dots,\, \frac{g+1-k}{g+1}}_{k \text{ or } k-1 \text{ times}}\Bigg)\!\! \in \RR^g\Biggr\}.
    \end{equation}
  The $f$-vector $(f_0, f_1, \dots, f_{g-1})$ of $V_{Q}$ is given by $f_l = \binom{g+1}{l}(2^{g+1-l}-2)$ for $l = 0, \dots, g-1$.
\end{theorem}

\begin{proof}
    Define the linear map
\[
\Phi:\RR^g\longrightarrow H\subset\RR^{g+1},\qquad
\Phi(\xx)=\bigl(x_1,\dots,x_g,-\sum_{i=1}^g x_i\bigr),
\]
where $H=\{{\bf u}\in\RR^{g+1}:\sum_{i=1}^{g+1}u_i=0\}$ is the 
hyperplane orthogonal to $\mathbf 1=(1,\dots,1)$.
Then for all $\xx,\yy\in\RR^g$,
\[
\Phi(\xx)\cdot \Phi(\yy)=x\cdot y +\bigl(\sum_{i=1}^g x_i\bigr)\bigl(\sum_{i=1}^g y_i\bigr)=\langle \xx,\yy\rangle_Q,
\]
so $\Phi$ is an isometric isomorphism $(\RR^g,\langle\cdot,\cdot\rangle_Q)\xrightarrow{\cong}(H,\text{Euclidean})$.
Moreover, $\Phi(\ZZ^g)=\Lambda_H$, where $\Lambda_H:=\{{\bf u}\in\ZZ^{g+1}:\sum u_i=0\}$ is a rank-$g$ sublattice of $H$. In the notation of Section \ref{sec:OSbackground}, for the banana graph we have
$
C_1(\Gamma,\RR) \cong \RR^{g+1},$ and $
\Lambda := C_1(\Gamma,\ZZ) \cong \ZZ^{g+1}.
$
Moreover,
\[
H_1(\Gamma,\RR) = \ker \partial
   = \biggl\{u \in \RR^{g+1} : \sum_{i=1}^{g+1} u_i = 0\biggr\} = H,
\]
so that
$
H_1(\Gamma,\mathbb{Z}) = \Lambda \cap H = \Lambda_H.
$ Thus $H$ is precisely $H_1(\Gamma,\RR)$, and the orthogonal projection
$
\pi_H : C_1(\Gamma,\RR) \to H
$ along the line spanned by $\mathbf{1}$ coincides with the projection $\pi'$ from \eqref{eq:rho_pi_prime} onto $H_1(\Gamma,\mathbb{R})$ under the above identification.
Let $V_{\mathrm{std}}\subset\RR^{g+1}$ denote the Euclidean Voronoi cell of $\ZZ^{g+1}$ at $\bf 0$, namely the cube $C=C^{g+1}$ as defined in~\eqref{eq:cube}. A standard Voronoi inequality calculation shows that the Euclidean Voronoi cell of the sublattice $\Lambda_H$ inside $H$ is the orthogonal projection of $C$ to $H$:
\[
\mathrm{Vor}_H(\Lambda_H)\coloneqq V_{I_{g+1}}\cap \ker \Delta_1 =\pi_H(C),
\]
where $\pi_H$ is the orthogonal projection along $\mathbf 1$ onto $H$. Indeed, the defining halfspaces for $\mathrm{Vor}_H(\Lambda_H)$ come from the inequalities ${\bf m}\cdot {\bf u}\le \frac12 \|{\bf m}\|^2$ for all ${\bf m}\in\Lambda_H\setminus\{\bf 0\}$, which are the restrictions to $H$ of the corresponding inequalities for $\ZZ^{g+1}$. Since $\Phi$ is an isometry $\RR^g\cong H$, we obtain
\[
V_Q\,=\,\Phi^{-1}\bigl(\mathrm{Vor}_H(\Lambda_H)\bigr)\,=\,\Phi^{-1}\bigl(\pi_H(C)\bigr).
\]
Thus $V_Q$ is the orthogonal projection of the cube $C$ to $H$ transported back to $\RR^g$ via $\Phi^{-1}$.
In particular, $V_Q$ has dimension $g$.

The vertices of $\pi_H(C)$ are precisely the projections of the vertices of $C$ except the two antipodal vertices $\pm\frac12\mathbf 1$, which both project to ${\bf 0}\in H$. Let $\vv\in\{\pm\tfrac12\}^{g+1}$ be a cube vertex and let $t$ be the number of $+{\tfrac12}$ entries of $\vv$.
The orthogonal projection to $H$ is
\[
\pi_H(\vv)\,=\,\vv-\frac{\vv\cdot \mathbf 1}{\|\mathbf 1\|^2}\,\mathbf 1
\,=\,\vv-\frac{2t-(g+1)}{2(g+1)}\,\mathbf 1.
\]
Hence each $+$ coordinate of $\vv$ becomes $\frac{g+1-t}{g+1}$ and each $-$ coordinate becomes $-\frac{t}{g+1}$ in $\pi_H(\vv)$.
Thus, for $t=k\in\{1,\dots,g\}$, the projected vertex in $H$ has $k$ coordinates equal to $\frac{g+1-k}{g+1}$ and $g+1-k$ coordinates equal to $-\frac{k}{g+1}$, with total sum $0$.
Identifying $H\cong \RR^g$ by deleting any one coordinate (which preserves vertices), we obtain exactly the family of $g$-vectors listed in \eqref{eq:bananavorvert}, with the two possible multiplicities “$g-k$ or $g+1-k$” and “$k$ or $k-1$” depending on whether the deleted coordinate was a positive or negative entry.

Faces of $\pi_H(C)$ are images of faces of $C$ under the projection along $\RR\mathbf 1$. An $(\ell+1)$-dimensional face $F$ of $C$ is obtained by choosing a set $S$ of $\ell+1$ coordinates to vary in $[-\tfrac12,\tfrac12]$ and fixing the remaining $g+1-(\ell+1)=g-\ell$ coordinates to signs $\pm\tfrac12$. The direction space of $F$ is $\mathrm{span}\{{\bf e}_i:i\in S\}$, which does not contain $\mathbf 1$ unless $S=\{1,\dots,g+1\}$. Hence for every proper face ($|S|\le g$), the projection $\pi_H$ is injective on the affine span of $F$ and lowers dimension by exactly $1$, so $\pi_H(F)$ is an $\ell$-dimensional face of $\pi_H(C)$.

To count $\ell$-faces of $\pi_H(C)$ we may count $(\ell+1)$-faces of $C$ and subtract those whose projection degenerates. Degeneracy occurs precisely when the fixed coordinates are all $+\tfrac12$ or all $-\tfrac12$, because in those two cases $F$ lies in a facet orthogonal to $\mathbf 1$ and its projection collapses by one further dimension. 
Hence
\[
f_\ell=\binom{g+1}{\ell}\bigl(2^{\,g+1-\ell}-2\bigr), \qquad \ell=0,1,\dots,g-1.
\]
Transporting $\pi_H(C)$ back to $\RR^g$ via the isometry $\Phi^{-1}$ gives $V_Q$, preserving the face lattice and the coordinates given above.

\end{proof}
In the case of banana graphs, the Delaunay-translation equivalence classes can be described explicitly. This justifies the notation introduced in~\Cref{thm:vertices_voronoi_Rg}, as formalized in the following result.
\begin{proposition}\label{prop:vertexequiv}
    Let $V_Q$ be the Voronoi polytope associated with a genus $g$ banana graph. For two vertices $\aa ,\aa'\in {\cal V}_Q$ as in \Cref{thm:vertices_voronoi_Rg}, then $\aa \sim \aa'$ if and only if there exists $k\in\{1,\dots,g\}$ such that $\aa,\aa'\in [{\bf k}]$.
\end{proposition}

\begin{proof}
    Let $k$ be such that $\aa\in [{\bf k}]$, i.e.,  $\aa$ has coordinates equal to either $-\frac{k}{g+1}$ or $\frac{g+1-k}{g+1}$. Then, any integer translate of $\aa$ has coordinates with absolute value greater than one (meaning it lies outside of the Voronoi polytope) or equal to either $-\frac{k}{g+1} +1 = \frac{g+1-k}{g+1}$ or $\frac{g+1-k}{g+1} - 1 = -\frac{k}{g+1}$. Thus any vertex equivalent to $\aa$ must also also lie in the class $[\bf k]$. The other direction is straightforward as it is enough to check that if $\aa$ and $\aa'$ lie in the same class $[{\bf k}]$, then their coordinates can only differ by $0, 1, -1$. 
\end{proof}

Thus, the set of Voronoi vertices ${\cal V}_Q$ is divided into $g$ equivalence classes $[{\bf 1}],\dots,[{\bf g}]$, so that the class $[{\bf k}]$ contains $\binom{n}{k}$ points.

\begin{example}[Genus $2$]
    In the notation of~\Cref{thm:vertices_voronoi_Rg}, the equivalence classes described in \Cref{ex:sec1_genus2_banana} and \Cref{fig:hexagon} become $[\textcolor{myblue}{\bf 1}]=\{\aa_1,\aa_2,\aa_3\}$ and   $[\textcolor{mygreen}{\bf 2}]=\{\aa_4,\aa_5,\aa_6\}$.
\end{example}

We now recover the combinatorial structure of the polytope obtained by pulling back the Voronoi polytope $V_Q$ to $\RR^n$ via the canonical surjection $\rho:C_1(\Gamma,\RR)\to H^1(\Gamma,\RR)$, introduced in \Cref{sec:OSbackground}, represented by the cycle basis matrix $B$. Recall that two polytopes are said to be \textit{combinatorially equivalent} if there exists a bijection between their faces that preserves the inclusion relation; see \cite{ziegler2012lectures}. 

\begin{proposition}\label{prop:im_proj_hypercube}
    The polytopes $V_Q$ and $B^T(V_Q)$ are combinatorially equivalent. Moreover, $B^T(V_Q)$ has vertex set $\bigcup_{i=1}^gB^T([\kk])$, where for each $k = 1, \dots, g$, 
\begin{equation*}
B^T\!([\kk]) \coloneqq\Biggl\{\text{permutations of }  \Bigg( \underbrace{ \frac{k}{g+1}, \dots,  \frac{k}{g+1}}_{n-k \text{ times}}, \underbrace{-\frac{g+1-k}{g+1},\dots, -\frac{g+1-k}{g+1}}_{k \text{ times}}\Bigg)\!\Biggr\} \subset \RR^{n}.
 \end{equation*}
Moreover, $B^T(V_Q)$ can be realized as the projection of the $n$-cube $C^{n}=[-\tfrac{1}{2},\tfrac{1}{2}]^{n}$ along the all-ones vector ${\bf 1} \in \RR^{n}$ via the linear map $\pi': \RR^{n}\rightarrow \RR^{n}$ \hbox{represented by the matrix with entries} 
\[
\pi'_{ij} = \begin{cases}
    \frac{g}{g+1}, & \text{if }i=j\\
    -\frac{1}{g+1}, & \text{if }i \neq j.
\end{cases}
\]
\end{proposition}
\begin{proof}
   The shape of the vertices of the polytope $B^T(V_Q)=\text{Conv}(B^T\aa, \aa\in {\cal V}_Q)$ given in~\Cref{prop:im_proj_hypercube} follows directly from applying  $B^T$ to the vertices of $V_Q$ given in \eqref{eq:bananavorvert}.

   Let $\vv=(-\tfrac{1}{2},\dots,-\tfrac{1}{2},\tfrac{1}{2},\dots,\tfrac{1}{2})\in \RR^{n}$ be the vertex of the cube $C^n$ with $n-k$ entries equal to $-\tfrac{1}{2}$ and $k$ entries $\tfrac{1}{2}$. It is straightforward to check that $\pi'(\vv)=B^T \aa$ with $\aa$ as in \eqref{eq:bananavorvert}. All vertices in \Cref{prop:im_proj_hypercube} and \eqref{eq:bananavorvert} can be recovered by taking all permutations of the entries in $k$ for $k=1,\dots,g.$ 
   Note that $\pi'(\pm(\tfrac{1}{2},\dots, \tfrac{1}{2}))={\bf 0}$; so their image lies in the interior of $B^T(V_Q)$ and does not contribute as vertex.
\end{proof}

\subsection{Delaunay polytopes are uniform matroid polytopes}
With the combinatorics of the Voronoi polytope understood, we turn to its dual objects: the Delaunay polytopes. For banana graphs, we find these are hypersimplices whose combinatorics depends only on the equivalence class of the corresponding Voronoi vertex. 

Recall that $n=g+1$. The $g$-dimensional hypersimplex $\Delta_{k,n}\subset \RR^n$ is defined by
$$\Delta_{k,n} \, = \, \text{Conv} \biggl(\sum_{i \in I} \ee_i: I \in \binom{[n]}{k}\biggr),$$
where $\ee_i$ is the $i$-th standard basis vector in $\RR^n$.
This polytope is the matroid basis polytope of the uniform matroid $U_{k, n}$; equivalently, it is the top-dimensional matroid cell of the Grassmannian $\Gr(k, n)$. For the relevant background on matroids, we refer to \cite{oxley2006matroid}. In what follows, we show that the Delaunay polytopes associated with banana graphs are hypersimplices. 
\begin{theorem}\label{thm:BTD}
   Let $\aa\in [{\bf k}]\subset {\cal V}_Q$, and  $D_{\aa,Q}$ be the associated Delaunay polytope. Then the polytope $B^T(D_{\aa,Q})$ is combinatorially equivalent 
   to the hypersimplex $\Delta_{k, n}$.
\end{theorem}
\begin{proof}
By \Cref{cor:conv_ak}, for any $\aa\in [{\bf k}]$ we have 
$${\cal D}_{\aa,Q}\, =\, \aa - [{\bf k}] \, =\, \{\aa-\aa_i, \,\, \aa_i\in[\kk]\}.$$ 
Applying $B^T$ yields $B^T(\DaQ) = B^T\aa - B^T([\kk])$, where the elements of $B^T([\kk])$ are described in \Cref{prop:im_proj_hypercube}. Let $B^T\bar{\aa}$
be the unique element of $B^T([\kk])$ whose first $n-k$ coordinates are positive. Then
\begin{equation*}
B^T(\mathcal{D}_{\bar{\aa},Q}) \, = \, \{{\bf 0}\} \,\,\bigcup\,\, \biggl\{ \sum_{\substack{i\in I\\I\subseteq[n-k]}}
\ee_i-\sum_{\substack{j\in J\\J\subseteq[n]\setminus[n-k]}} \ee_j, \quad 1\leq |I|=|J|\leq \min\{k,n-k\}\biggr\},
\end{equation*}
where, for a positive integer $m$, we write $[m] \coloneqq \{1,2,\dots,m\}$, and $\ee_i$ denotes the standard basis of $\RR^n$.
Define 
\[{\bf s}_{\bar{\aa}} = \sum_{i\in [n]\setminus[n-k]} \ee_i.\] 
A direct computation shows that the translated polytope $B^T(D_{\bar{\aa},Q}) + {\bf s}_{\bar{\aa}}$ is precisely the canonical hypersimplex $\Delta_{k, n}$. Choosing a different representative $\aa_i \in [\kk]$ simply permutes the coordinates of the vector $\mathbf{s}_{\bar{\aa}}$. In particular, $B^T(D_{\aa_i,Q}) + {\bf s}_{\aa_i}=\Delta_{k, n},$ where the vector $\mathbf{s}_{\aa_i}\in\NN^n$ has coordinates 
\begin{equation}\label{eq:s_vector}({\bf s}_{\aa_i})_{j} \,=\,
 \begin{cases}
     0 & \text{if }\,(B^T\aa_i)_j >0\\
     1 & \text{if }\,(B^T\aa_i)_j <0.
 \end{cases}  
\end{equation}
In addition, note that ${\bf s}_{\aa_i} = {\bf s}_{\bar{\aa}} + B^T\bar{\aa}-B^T\aa_i.$
\end{proof}

\begin{corollary}\label{cor:Delaunay_in_Rg}
    Let $\aa\in [\kk]$ be a vertex of the Voronoi polytope $V_{Q}$ in the form given in \eqref{eq:bananavorvert}. The Delaunay polytope $D_{\aa,Q}$ associated with $\aa$ is combinatorially equivalent to the hypersimplex 
$\Delta_{k,n}$, which is the matroid base polytope of the uniform matroid $U_{k, n}$.
\end{corollary}
\begin{proof}
    This follows from the fact that the affine map $\rho$ defined in \Cref{sec2} is invertible, when restricted to the subspace $\{x_1+\dots+x_{n}=0\}\subset\RR^{n}$ containing the polytope $B^T(D_{\aa,Q})$.
\end{proof}

\begin{example}[Genus $2$]
    The polytope $B^T(V_Q)$ is the hexagon in $\RR^3$, contained in the subspace defined by $x_1 + x_2 + x_3 = 0$, with vertices 
     \begin{equation}\label{eq:BT(V_Q)_genus2}
    \begin{aligned}
     B^T\aa_1 &= \left(\tfrac{1}{3}, \tfrac{1}{3}, -\tfrac{2}{3}\right), &
     B^T\aa_2 &= \left(-\tfrac{2}{3}, \tfrac{1}{3}, \tfrac{1}{3}\right), &
     B^T\aa_3 &= \left(\tfrac{1}{3}, -\tfrac{2}{3}, \tfrac{1}{3}\right),&
     \\[2pt]
   B^T\aa_4 &= \left(\tfrac{2}{3}, -\tfrac{1}{3}, -\tfrac{1}{3}\right), &
   B^T\aa_5 &= \left(-\tfrac{1}{3}, \tfrac{2}{3}, -\tfrac{1}{3}\right), & 
   B^T\aa_6 &= \left(-\tfrac{1}{3}, -\tfrac{1}{3}, \tfrac{2}{3}\right).
\end{aligned}
\end{equation}
One verifies that $B^T(V_Q)$ is the image of the cube $C^3$ under the projection described in~\eqref{eq:pi_prime_example}. This is illustrated in \Cref{fig:voronoibananag3} (left). Furthermore, the Voronoi vertices in \Cref{ex:sec1_genus2_banana} give rise to the six triangular Delaunay polytopes shown in \Cref{fig:hexagon}. For each $i=1,2,3$, the Delaunay polytope $B^T (D_{\aa_i,Q}) + \mathbf{s}_{\aa_i}$ is  the hypersimplex $\textcolor{myblue}{\Delta_{1,3}}$, whereas for $i=4,5,6$ the same construction yields the hypersimplex $\textcolor{mygreen}{\Delta_{2,3}}$.
\end{example}

Notice that the polytope $B^T(D_{\aa,Q})$ can still be interpreted as the Delaunay polytope associated to the point $B^T \aa\in \RR^{n}$ with respect to the Euclidean metric induced by the identity matrix $I_n$ restricted to the sublattice $\ker \Delta_1$. 

\begin{proposition}
Let $\aa \in \RR^g$ be a Voronoi vertex of the Voronoi polytope $V_Q$ associated with the banana graph. Then $B^T \DaQ \, = \, \mathcal{D}_{B^T\aa,I_{n}} \cap \ker \Delta_1.$
\end{proposition}

\begin{proof}
  Let $\cc\in \DaQ$. A direct computation shows that $B^T\cc \in \mathcal{D}_{B^T\aa,I_{n}}$. Furthermore, by~\eqref{eq:rho_pi_prime} we have $\pi' \circ B^T = B^T$, hence $B^T\cc\in \ker\Delta_1.$
  
For the reverse inclusion, recall that $B^T\aa = \pi' (\vv)$, where $\vv$ is a vertex of the $n$-dimensional hypercube from~\eqref{eq:cube}. Let $B^T\bar{\aa}\in B^T([\kk])$ be the element whose first $n-k$ coordinates are positive; then $B^T\bar{\aa}=\pi'(\vv)$, where
$$\vv\, =\, \Bigg( \underbrace{ \frac{1}{2}, \dots,  \frac{1}{2}}_{n-k \text{ times}}, \underbrace{-\frac{1}{2},\dots, -\frac{1}{2}}_{k \text{ times}}\Bigg).$$
Let $\tilde{\cc}\in \mathcal{D}_{B^T{\bar{\aa}},I_{n}}\cap \ker \Delta_1$. Then $\tilde{\cc}\in {\cal D}_{\pi'(\vv),I_{n}}$ and, since $\tilde{\cc}=\pi'(\tilde{\cc})$, we have $\tilde{\cc} = (\tilde{c}_1,\tilde{c}_2,\dots, \tilde{c}_n)$  with
$$\tilde{c_i}\in \{0,1\}\quad  \text{for } i\in [n-k],\quad  \text{and} \quad \tilde{c_j}\in \{0,-1\} \quad \text{for } j\in [n]\setminus[n-k],$$
and the additional condition that the number of positive entries equals the number of negative entries. As shown in the proof of \Cref{thm:BTD}, this precisely characterizes the set of vectors in $B^T \DaQ$. This concludes the proof.
\end{proof}

\subsection{Strongly connected orientations and matroids}

The combinatorial structures described so far have a natural interpretation in terms of graph orientations, which we discuss in this section. Our goal here is to use the canonical bijection between equivalence classes of Voronoi vertices and equivalence classes of strongly connected orientations to associate a matroid with each vertex of the graph $\Gamma$.

\begin{definition}
    Let $(\Gamma,\iota)$ be a connected oriented metric graph. The orientation $\iota$ is \textit{strongly connected} if any pair of vertices $u,v$ are connected by an oriented path from $u$ to $v$ and an oriented path from $v$ to $u$. We denote by $\mathcal{O}(\Gamma)$ the set of strongly connected orientations\hbox{ of $\Gamma$.}
\end{definition}

For the class of banana graphs considered here, an orientation is strongly connected if and only if at every vertex $v$, there exist edges $e,e'$  such that $e_+=v$ and $e'_-=v.$

The next result gives an explicit bijection between the strongly connected orientations of a banana graph $\Gamma$ and the set of vertices of its Voronoi polytope $V_Q$. A general version of this correspondence appears in \cite{amini2010lattice}, where the face poset of a Voronoi cell is shown to be isomorphic to the poset of strongly connected orientations for any finite connected graph. 

\begin{theorem}\label{thm:bijection}
    Let $(\Gamma,\iota)$ be an oriented banana graph of genus $g$. Then there exists a canonical bijection 
    between the set of vertices $\mathcal{V}_Q$ and the set $\mathcal{O}(\Gamma)$ of strongly connected orientations of $\Gamma$ given by
    \begin{align}\label{eq:phi}
    \phi\,:\,\mathcal{V}_Q&\to \mathcal{O}(\Gamma)\\
    \,\,\,\aa\,\,\, &\mapsto \iota_\aa(e_i) = \begin{cases}
        \iota(e_i) &\text{if } (B^T\aa)_i>0,\\
        \iota(-e_i) & \text{if } (B^T\aa)_i<0.\nonumber
    \end{cases}
    \end{align}
\end{theorem}

\begin{proof}
For clarity, we fix once and for all an initial orientation $\iota_0$ as in~\Cref{fig:bananagen3}. By \eqref{eq:rho_pi_prime}, we have $B^T(V_Q) = \pi'(C^{n})$, so the vertices of $B^T(V_Q)$ are the vectors $\pi'(\vv)$ with $\vv\in C^n.$ For a vector ${\bf w}$, let $\sigma({\bf w})$ denote the vector of signs its coordinates. Note that $\sigma(\pi'(\vv))=\sigma(\vv)$ for all $\vv\in C^n,$ and every such signature contains at least one positive and one negative entry, and no zero entries. Given $\aa\in \mathcal{V}_Q$, choose $\cc\in C^n$ such that $B^T\aa=\pi'(\cc)$. Then $\sigma(B^T\aa)=\sigma(\cc)$. Define the orientation $\iota_\aa$ via~\eqref{eq:phi}, namely $\sigma(B^T\aa)_i=\text{sign}(\langle B^T\aa,e_i\rangle)$ determines whether the edge $e_i$ is oriented as $\iota_0(e_i)$ or as its reverse.

It follows directly from the description of $C^n$ that $\phi$ sets up a bijection between $\mathcal{V}_Q$ and ${\cal O}(\Gamma$). Finally, it is straightforward to verify that the resulting orientation $\iota_\aa$ is  independent of the choice of the initial orientation $\iota_0$, and thus the bijection is canonical. 
\end{proof}

The next result translates the equivalence relation (under Delaunay-translation) in ~\Cref{eq:equivalence_Voronoi_vertices} on the set of Voronoi vertices ${\cal V}_Q$ to the set of strongly connected orientations~${\cal O}(\Gamma)$. 

\begin{corollary}\label{cor:circuits}
    Two vertices $\aa,\aa'\in {\cal V}_Q$ are equivalent, that is $\aa\sim \aa'$ with $\aa,\aa'\in[\kk]$ for some $k\in[g]$, if and only if the corresponding orientations $\iota_{\aa}$ and $\iota_{\aa'}$ have $k$ outgoing edges at $v_1$, equivalently the edges which change orientation in $\iota_{\aa'}$ from $\iota_{\aa}$ form a circuit in $\iota_{\aa}$.
\end{corollary}

\begin{proof}
The first part of the statement is an immediate application of the bijection presented in Theorem \ref{thm:bijection}, which shows that $\aa,\aa'\in [\kk]$ if and only if $\iota_{\aa}$ and $\iota_{\aa'}$ have $k$ outgoing edges. Since $\aa, \aa'$ are distinct, there must be an outgoing edge in  $\iota_{\aa}$ which is incoming in $\iota_{\aa'}$ and vice versa, Moreover, the edges that switch orientations in $\iota_{\aa'}$ versus $\iota_{\aa}$ must have an equal number outgoing from $v_1$ as incoming. Together these edges make a circuit in $\iota_{\aa}$ which changes orientation in $\iota_{\aa'}$. On the other hand, given a circuit in $\iota_{\aa}$, it consists of an equal number of incoming and outgoing edges from $v_1$. Swapping the orientation of this circuit maintains the number of edges outgoing from $v_1$, and thus the equivalence class.
\end{proof}

An immediate consequence of~\Cref{cor:circuits} is that, given a vertex $\aa\in[{\bf k}]$, the number $k$ not only determines that $\DaQ$ is combinatorially equivalent to the hypersimplex $\Delta_{k,n},$ but also encodes the number of outgoing edges in the corresponding strongly connected orientation $\iota_\aa$: there are $k$ edges leaving $v_1$ and $n-k$ leaving $v_2$ (equivalently, $n-k$ entering $v_1$). This property holds true for every element of the equivalence class $[\kk]$.

We now assign a matroid $\mathcal{M}_{\aa,v_i}$ to each representative $\aa\in [\kk]$ and to each choice of vertex $v_i$ of the graph $\Gamma$ for $i=1,2$. These matroids are defined via their sets of bases as follows. Let ${\bf s}_\aa$ be the shift vector defined in \eqref{eq:s_vector}. For the vertex $v_1$, we define the matroid $\mathcal{M}_{\aa,v_1}$ on the ground set $[n]$ with the set of bases

\begin{equation} \label{eq:matroid_v1}
\mathcal{B}_{\aa,v_1} \,\coloneqq\, \left\{ I \in \binom{[n]}{k}\,\,|\,\,  \exists \cc\in \DaQ \text{ such that } (B^T\cc + \mathbf{s}_{\aa})_i = 1 \text{ for all } i \in I \right\}.
\end{equation}
Dually, for the vertex $v_2$, the matroid $\mathcal{M}_{\aa,v_2}$ has as its bases the complements of the above, given by
\begin{equation} \label{eq:matroid_v2}
\mathcal{B}_{\aa,v_2} \,\coloneqq\, \left\{ \bar{I} \in \binom{[n]}{n-k}\,\,|\,\,  \exists \cc\in \DaQ \text{ such that } (B^T\cc + \mathbf{s}_{\aa})_i = 0 \text{ for all } i \in I \right\},
\end{equation}
where $\bar{I}=[n]\setminus I$. Clearly, from~\Cref{thm:BTD}, we have that, for $\aa\in[{\bf k}]$, the matroid $\mathcal{M}_{\aa,v_1}$ (resp. $\mathcal{M}_{\aa,v_2}$) is isomorphic to $U_{k,n}$ (resp. $U_{n-k,n}$).

In the special cases $k=1$ or $k=g$, the strongly connected orientations coincide with the \emph{perfect orientations} in the Postnikov plabic graph, see \cite{Pos}, obtained by identifying the vertex $v_1$ (or, dually $v_2$) with the boundary of the disk and coloring the vertex $v_2$ (or, dually $v_1$) white or black, respectively.
More generally, our choice of matroid bases is guided from the analogy with Postnikov's \textit{boundary measurement map}, when applicable. In the setting of \cite{Pos}, oriented planar bicolored graphs with $k$ outgoing edges and $n-k$ incoming edges at the boundary of the disk describe points in the \textit{totally nonnegative Grassmannian} $\Gr(k,n)_{\geq 0}$.

Given a vertex $\aa\in[{\bf k}]\subset {\cal V}_Q$, let $(\Gamma,\iota_\aa)$ be the oriented banana graph described by the bijection in \eqref{eq:phi}. In the next theorem, we describe how to think of the bases of the matroid $\mathcal{M}_{\aa,v_1}$ in terms of the Delaunay points associated to $\aa$.

\begin{theorem}\label{thm:delsetismatroid}
    Let $\aa\in [{\bf k}] \subset \mathcal{V}_Q$, and $\iota_\aa$ the strongly connected orientation of $\Gamma$ associated with it. Then the Delaunay points associated to $\aa$ are in bijection with elements of $[{\bf k}]$ and bases of the matroid $\mathcal{M}_{\aa,v_1}$.
\end{theorem}

\begin{proof}
    As a consequence of Corollary \ref{cor:circuits} and the correspondence between Delaunay points and circuits proved in \cite{froman}, we find that the Delaunay points associated to $\aa$ are in bijection with strongly connected orientations of $\Gamma$ which have exactly $k$ edges outgoing from $v_1$. It then follows that this is in bijection with the set of bases of $U_{k, n} = \mathcal{M}_{\aa,v_1}$.
\end{proof}

We remark that Theorem \ref{thm:delsetismatroid} may not seem profound, as the correspondence between vertices in an equivalence class and bases of the associated matroid is rather obvious. This is a feature of our restriction to banana graphs. The more general version of this theorem, which we leave to forthcoming work, is more delicate but carries the same spirit and methods of proof. It is beyond the scope of this paper, however we include the following definition to foreshadow the general statement.

\begin{definition}
    Let $\Gamma$ be a metric graph and $\iota$ a strongly connected orientation. Let $\mathcal{O}$ be the set of strongly connected orientations of $\Gamma$ which have the same in and out degrees at each vertex as $\iota$. Fix a vertex $v$ of $\Gamma$, and label the edges adjacent to $v$ by $1 \dots \ell$.

    The \emph{Delaunaytroid} $\mathcal{D}_{\Gamma, v, \iota}$ is the matroid whose set of bases is $\{B_\mathcal{o}: \mathcal{o} \in \mathcal{O}\}$, where $B_\mathcal{o}$ is the set of labels of edges outgoing from $v$ in the orientation $\mathcal{o}$.
\end{definition}

\begin{example}[Genus $2$]
 We can label the vertices of the Voronoi polytope $V_Q$ with the sign patterns given by the corresponding vertices in $B^T(V_Q)$ described in \eqref{eq:BT(V_Q)_genus2}. The sign pattern induces a strongly connected orientation of the corresponding banana graph as explained in \Cref{thm:bijection} and illustrated in \Cref{fig:hexagon_orientations} (right).
\begin{figure}[h]
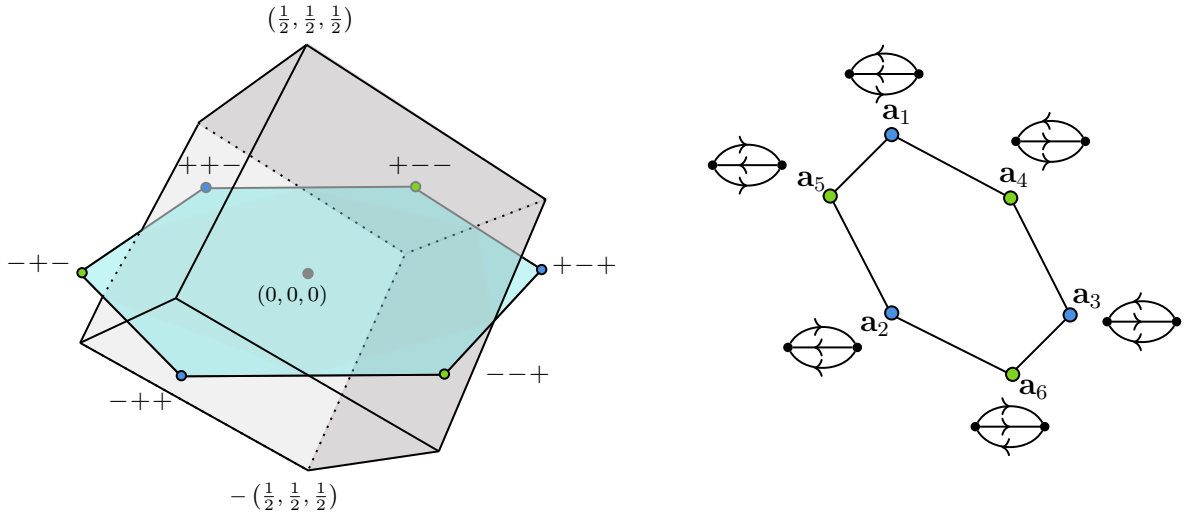

\captionsetup{width=.8\linewidth, font=small}
  \begin{subfigure}{.5\textwidth}
  \centering    
  \includegraphics[]{projection.tikz}
  \end{subfigure}
  \begin{subfigure}{.5\textwidth}
  \centering
\includegraphics[]{hexagon2_orientations.tikz}
  \end{subfigure}
  \caption{Left: the hexagon $B^T(V_Q) \subset \mathbb{R}^3$ associated with the graph $\Gamma_2$, obtained as the projection of the cube $C^3$ (in gray) via the map in \eqref{eq:pi_prime_example}. The cube's vertices $\pm(\frac{1}{2},\frac{1}{2},\frac{1}{2})$ are projected to the origin. The sign patterns of the vertices of the hexagon determine the isomorphism from \Cref{thm:bijection} between Voronoi vertices and strongly connected orientations of $\Gamma_2$ (right).}  \label{fig:hexagon_orientations}
\end{figure}
\end{example}

\begin{example}[Genus $3$]
The Voronoi polytope $V_Q$ associated to the banana graph in genus~$3$ is the rhombic dodecahedron, see \Cref{fig:voronoibananag3} and \cite[Figure~4]{AgoCelStrStu}. Its $f$-vector is $(14,24,12)$, where the $14$
vertices split into three equivalence classes $[\textcolor{myblue}{{\bf 1}}],[\textcolor{mygreen}{{\bf 2}}]$ and, $[\textcolor{myred}{{\bf 3}}]$, of cardinality $4,6,$ and $4$, respectively. The Delaunay sets arising for the elements of each class are combinatorially equivalent to the hypersimplices $\textcolor{myblue}{\Delta_{1,4}},\textcolor{mygreen}{\Delta_{2,4}},$ and $\textcolor{myred}{\Delta_{3,4}}$. The isomorphism given in \Cref{thm:bijection} is illustrated in \Cref{fig:voronoibananag3}.
\begin{figure}[ht]
    \centering
    \captionsetup{width=.8\linewidth, font=small}
\includegraphics{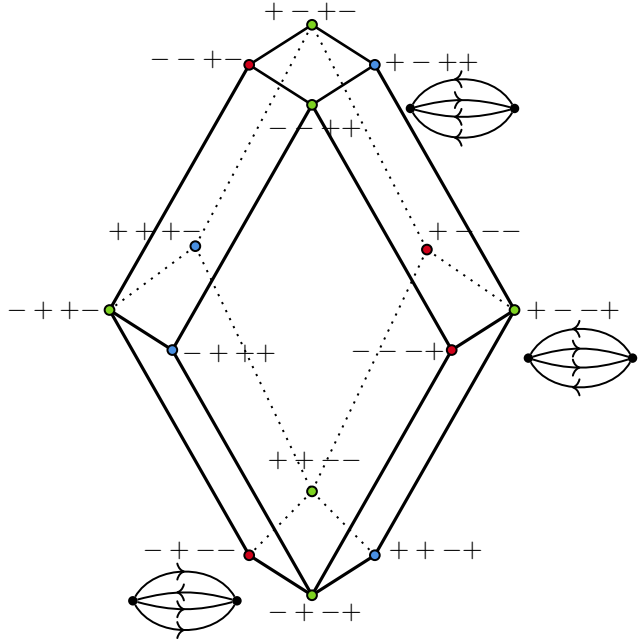}
    \caption{The Voronoi cell for the banana curve of genus $3$ with vertices labeled by strongly connected orientations of $\Gamma_3$. Note that, while the signatures depend on the initial orientation of the graph, the final orientation is independent of such a choice.} 
    \label{fig:voronoibananag3} 
\end{figure}
\end{example}

These identifications complete the dictionary between the tropical Jacobians of the banana graphs and the combinatorics of hypersimplices and uniform matroids. In Section \ref{sec:Hirota_of_graph} we use this dictionary to translate tropical theta functions into KP $\tau$–functions and to describe the associated Hirota varieties.

\section{KP multi-solitons and limit theta functions}\label{sec:theta_functions_limit}

The KP equation admits a rich family of algebro–geometric (finite–gap) solutions arising from smooth complex curves. Let us start recalling the algebro--geometric data necessary to construct complex smooth KP solutions following \cite{Krich1976,Krich1977} and \cite{dubrovin}. 

Fix the triple $(\X,p_0,\zeta)$, where $\X$ is a smooth complex algebraic curve of genus $g$, $p_0$ is a distinguished point on $\X$, and $\zeta$ is the affine coordinate such that $\zeta^{-1}(p_0)=0$. Then, for any given degree $g$ non--special divisor $p_1+\dots +p_g\in \text{Div}(\X)$ associated with $(\X,p_0,\zeta)$, there exists a unique (up to normalization at $p_0$) the Baker--Akhiezer function is analytic on $\X\setminus \{ p_0,p_1,\dots,p_g \}$, with pole divisor in $p_1+\dots +p_g$ and essential singularity at $p_0$. This Baker--Akhiezer function solves a spectral problem whose compatibility condition is given by the KP equation \eqref{eq:KPequation}. Then the complex smooth finite-gap solutions to the KP equation on $(\X,p_0,\zeta)$ are parameterized by the divisor $p_1+\cdots+p_g$ as follows.

Let $Q$ be a Riemann matrix for $\X$, i.e. complex symmetric with negative definite real part.
 The classical Riemann theta function
\begin{equation}
\label{eq:RTFimaginary}
\theta(\mathbf z\,|\,Q)
\,=\,\sum_{\cc\in\mathbb{Z}^g}
\exp\!\left[\frac{1}{2}\,\cc^{T}Q\,\cc+\cc^{ T}\mathbf z\right]
\end{equation}
yields the \emph{finite-gap} solution associated with the algebro--geometric data $(\X, p_0, \zeta, p_1,\dots,p_g)$
\begin{equation}
\label{eq:kpthetasol}
u(x,y,t)\;=\;c+2\,\partial_x^2\log\theta\big(\mathbf U\,x+\mathbf V\,y+\mathbf W\,t+\boldsymbol{D}\,\big|\, Q),
\end{equation}
 where the vectors $\mathbf U,\mathbf V,\mathbf W$ are certain periods of normalized meromorphic differentials of the second kind with a pole at the distinguished point $p_0=\infty$, $c\in\mathbb C$ is a constant, and 
\begin{equation}\label{eq:Dparameter}
     \boldsymbol{D}\,=\,-\mathcal{A}(p_1,\dots,p_g)-\mathbf K
 \end{equation}
with $\mathcal{A}$ the Abel map and $\mathbf K$ the vector of Riemann constants. 
 
 The triple $(\mathbf U,\mathbf V,\mathbf W)$ traces a three–dimensional algebraic variety in the weighted projective space $\mathbb{W}\mathbb{P}^{3g-1}$, namely the Dubrovin threefold \cite{Agostini_2021}. The behavior of the Dubrovin threefold under degenerations of the underlying curve has been considered in \cite[Section 6]{Agostini_2021}. Subsequently, \cite{AgoFevManStu} introduced the Hirota variety which parametrizes solutions of the KP equation from tropical limits of the Riemann theta function, thereby capturing the limiting KP data in a combinatorial–geometric framework. 

 The tropical limits of finite-gap solutions are KP multi-soliton solutions whose data are encoded by Grassmannians. 
 In this work, we refine that picture by establishing an explicit correspondence between multi-soliton data and tropical $\tau$–functions with identifying the parametrizations of Hirota varieties for banana graphs. 
 
 In this section, we recall the basic framework of the KP multi-soliton solutions in their Grassmann form and we consider the tropical limit of the theta function and of the $\UU,\VV,\WW$ in the case of banana graphs. Then in \Cref{sec:Hirota_of_graph} we characterize the Hirota variety for banana graphs, identifying the KP multi-solitons that arise as the $\varepsilon\to 0$ tropical limits of finite-gap KP solutions on the hyperelliptic curves $X_\varepsilon$.

\subsection{KP multi-solitons}

Let $\{ \k_1,\dots, \k_n \}$ be either real or complex parameters. In the case that they are real, we shall assume that their ordering is $\k_1<\dots<\k_n$.
Let $A=(A_{ij})$ be either a real or complex $k\times n$ matrix of rank $k$. 
Following \cite{Abenda2018,FreemanNimmo,KW2013,Matveev}, we consider the $k$ functions
\begin{equation}\label{eq:f_functions}
    f^{(i)}(x,y,t)\,=\,\sum_{j=1}^n A_{ij}\,E_j(x,y,t),
    \qquad
    E_j(x,y,t)\,=\,\exp[\k_j x+\k_j^2 y+\k_j^3 t],
    \quad i\in[k],
\end{equation}
which form a $k$-dimensional space of solutions of the heat equations
\begin{equation}\label{eq:heat_hierarchy}
    \partial_y f\,=\,\partial_x^2 f,
    \qquad
    \partial_t f\,=\,\partial_x^3 f.
\end{equation}
The corresponding $\tau$–function is defined by the Wronskian $ {\rm Wr}_x(f^{(1)},\dots,f^{(k)})$, i.e,
\begin{equation}\label{eq:tauGrass}
    \tau_A(x,y,t) =\!\!\sum_{J\in \binom{[n]}{k}} \!\!
     A_J K_J E_J(x,y,t),
\end{equation}
where $A_J$ and $K_J$ denote the minors of the $k\times n$ matrix $A$ and the Vandermonde submatrix $(\k_j^{i-1})_{i\in[k],j\in[n]}$, respectively, taken with columns indices $J$. More explicitly, we have 
\begin{equation}\label{eq:KJ_EJ}
\displaystyle K_J\, =\, \prod_{\substack{i,j\in J\\i<j}}(\k_j-\k_i)
 \qquad \text{and} \qquad E_J(x,y,t) \,=\, \prod_{j\in J} E_j(x,y,t), 
\end{equation}
with the $E_j$ as in \eqref{eq:f_functions}. 
Then, the function 
\begin{equation}\label{eq:u_tau}
    u(x,y,t)\,=\,2\,\partial_x^2\log\tau(x,y,t)
\end{equation}
satisfies the KP equation~\eqref{eq:KPequation} and it is called a $(k,n)$--\emph{multi-soliton} because of its asymptotic properties in the space-time variables (see \cite{KW2013}). 
The minors $A_J$ are the Plücker coordinates of a point $A\in{\rm Gr}(k,n)$, so that the soliton data are naturally parametrized by the Grassmannian. 
The solution $u(x,y,t)$ is real and regular for all real $(x,y,t)$ if and only if all Plücker coordinates are non-negative, that is $A\in{\rm Gr}_{\geq 0}(k,n)$~\cite[Theorem~12.1]{KW2013} and $\k_1<\dots<\k_n$. In this case, the qualitative behavior of the solution depends on the positroid cell represented by $A$, as classified in~\cite{KW2014}. For convenience, we denote by $\tau_A(x,y,t)$ the $\tau$-function associated with the point $A \in \Gr(k,n)$.

 Given such expression for a $\tau$-function, a natural question is whether one can obtain a parametrization from an underlying nodal curve $X$, that is associate a $\tau$--function $\tau_{A}(x,y,t)$ for some matrix $A\in \Gr(k,n)$ starting from the combinatorics of $X$. This question will be studied in depth in~\Cref{sec:Hirota_of_graph}.

\subsection{Theta functions in the tropical limit}

We work with tropical curves in the tropicalization of the moduli space of complex smooth projective algebraic curves of genus $g$, denoted by $\mathcal{M}_g^{\rm tr}$. It consists of all metric graphs of genus $g$. A stable curve $\X$ in the Deligne-Mumford compactification of $\mathcal{M}_g$ is mapped to its dual graph in $\mathcal{M}_g^{\rm tr}$, namely the tropicalization ${\rm Trop}(\X)$ of $\X$. Similarly to the classical setting, we have the tropical Riemann matrix $Q$ coming from its tropical curve, which is a $g\times g$ real, symmetric and positive semidefinite matrix. The Riemann theta function becomes a finite sum of exponentials under the tropical degenerations. The combinatorial data that underlies the degenerate Riemann theta function is given by the Voronoi cells with respect to the lattice $\mathbb{Z}^g$ in the metric defined by $Q$. More precisely, we consider a Mumford curve $X$ of genus $g$. The tropical curve $\rm{Trop}(\X)$ can be read off from the tropical Riemann matrix. 

Now we consider the degeneration of $\X$ over $\mathbb{C}\{\!\!\{{\varepsilon}\}\!\!\}$ to its tropical limit $X$. The Riemann matrix can be written as 
\begin{equation}\label{eq:Repsilon}
    Q_\varepsilon\, =\, \log(\varepsilon^2)Q+R_\varepsilon
\end{equation}
where $R_\varepsilon$ is a $g\times g$ symmetric matrix with entries that are complex analytic functions in $\varepsilon$ that converge as $\varepsilon$ goes to $0$. We postpone to Section~\ref{sec:Riemannmatrix} the detailed description of this limit in the case of the banana graph. Fixing a point ${\aa}$, and replacing ${\zz}$ by ${\bf z}-\log(\varepsilon^2)Q{\bf a}$ in the classical Riemann theta function in~\eqref{eq:RTFimaginary}, we obtain

\begin{equation}\label{eq:Riemannthetafunctiondegeneration}
\theta({\bf z}-\log(\varepsilon^2)Q{\bf a}\,|\,Q_\varepsilon)
=\sum_{{\bf c}\in\mathbb{Z}^g}
\varepsilon^{\left({\bf c}^TQ{\bf c}-2{\bf c}^TQ{\bf a}\right)}
\exp\!\left[
\tfrac{1}{2}{\bf c}^TR_\varepsilon{\bf c}+{\bf c}^T{\bf z}
\right].
\end{equation}
This expression converges for $\varepsilon\rightarrow 0^{+}$ if $\aa$ belongs to the Voronoi cell $V_Q$. We state the following result from \cite[Theorem 3]{AgoFevManStu}, and  \cite[Theorem 4]{AgoCelStrStu} to present the limit: 

\begin{theorem}
    Fix ${\bf a}$ in the Voronoi cell of the tropical Riemann matrix $Q$. For $\varepsilon\rightarrow 0$, the series \eqref{eq:Riemannthetafunctiondegeneration} converges to a tropical theta function supported on the Delaunay set $\mathcal{D}=\mathcal{D}_{{\bf a},Q}$, namely
    \begin{equation}\label{eq:tropicalthetafunction}
        \theta_\mathcal{D}(\zz)=\sum_{\cc\in \mathcal{D}}a_\cc\exp\left[\cc^T\zz\right], \text{ where } a_\cc=\exp\left[\frac{1}{2}\cc^TR\cc\right],
    \end{equation}
    where $R =\displaystyle \lim\limits_{\varepsilon \to 0^+} R_\varepsilon$.
\end{theorem}

\begin{proof}
    The proof follows verbatim the arguments of \cite[Theorem 3]{AgoFevManStu} and \cite[Theorem 4]{AgoCelStrStu}.
The only difference is that we use a degeneration scaled by $\log(\varepsilon^{2})$ instead. This change yields 
the same asymptotic behavior: the dominant lattice points are unchanged and the limiting tropical theta function is identical.
We therefore omit the details.
\end{proof}

If ${\bf a}$ lies in the interior, the Delaunay polytope ${\rm conv}(\mathcal{D})$ is just the origin. On the other hand, if ${\bf a}$ is a vertex of the Voronoi polytope then the Delaunay polytope is $g$-dimensional. In fact, all intermediate dimensions for the Delaunay polytope can occur depending on the location of ${\bf a}$ in the Voronoi polytope.

We now study the behavior of the degenerate theta function $\theta_{\mathcal{D}_{\aa,Q}}(\zz)$ with respect to the choice of a different representative $\aa'$ in the same equivalence class of $\aa$ as in \Cref{eq:equivalence_Voronoi_vertices}.

\begin{lemma}\label{lem:equivalence_thetafunct}
    Let $\aa,\aa' \in \mathcal{V_Q}$ such that $\aa\sim\aa'$. Then the theta functions $\theta_{\mathcal{D}_{\aa,Q}}(\zz)$ and $\theta_{\mathcal{D}_{\aa',Q}}(\zz)$ coincide up to a factor of an exponential. 
\end{lemma}

\begin{proof}
    Let $\aa,\aa'\in \mathcal{V}_Q$ with $\aa\sim\aa'$. By \Cref{prop:equivvertex_general}, we have $\aa'=\aa-\cc_0$ for some Delaunay vertex $\cc_0\in \DaQ$. Evaluate the theta function \eqref{eq:tropicalthetafunction} at $\tilde{{\bf z}}=\zz-R\aa'$:
    \begin{align*}
    \theta_{\mathcal{D}_{\aa'\!,Q}}(\tilde{{\bf z}}) =& \sum_{\cc'\in \mathcal{D}_{\aa'\!,Q}} \exp\left[ \frac{1}{2} (\cc')^T R \cc' -(\cc')^T R \aa'\right] \exp[(\cc')^T \zz].\\
    =& \sum_{\cc\in \mathcal{D}_{\aa,Q}} \exp\left[ \frac{1}{2} (\cc-\cc_0)^T R (\cc-\cc_0) -(\cc-\cc_0)^T R (\aa-\cc_0)\right] \exp[(\cc-\cc_0)^T \zz],\\
    =&\exp\left[-\cc_0^T\zz -\frac{1}{2} \cc_0^T R\cc_0\right]\cdot \theta_{\DaQ}(\zz- R \aa),
     \end{align*}
    where in the second equality we used that $\mathcal{D}_{\aa',Q}= \DaQ - \cc_0$.
\end{proof}

\begin{example}[Genus $2$]\label{ex:theta_genus2_banana}
Let $\Gamma$ be the banana graph as in~\Cref{ex:sec1_genus2_banana}. Any choice of vertex $\aa_i$ of $V_Q$ leads to a triangle as its Delaunay polytope, e.g., $\mathcal{D}_{\aa_4,Q}=\{(0,0),(1,0),(0,1)\}$. Then, the tropical theta function is the sum $\theta_{\mathcal{D}_{\aa_4,Q}}(\zz)=a_{00}+a_{10}\exp(\zz_1)+a_{01}\exp(\zz_2)$.
\end{example}

Starting from the theta function of the point configuration ${\cal D} = {\cal D}_{\aa,Q}$ as described in \Cref{eq:tropicalthetafunction}, let us evaluate the theta function $\theta_{{\cal D}}(\zz)$ at 
\[\zz \,=\, \UU x + \VV y +\WW t + {\bf D},\]
where $\UU,\VV,\WW$ are coordinate vectors of size $g$ in the weighted projective space $\mathbb{W}\PP^{3g-1}$ with respective weights $1,2,3$, and ${\bf D} =(\log \beta_1,\dots,\log \beta_g)\in \CC^g$ with $\log$ being the natural logarithm:
\begin{equation}\label{eq:tau_from_theta}
    \tau_{\mathcal{D}}(x,y,t) \,  \coloneqq\, \theta_{{\cal D}}(\zz)\, = \,\sum_{\cc\in {\cal D}} \alpha_\cc \exp\biggl[ \big(\!\sum_{i=1}^g \cc_i U_i\big)x+\big(\!\sum_{i=1}^g \cc_i V_i\big)y+\big(\!\sum_{i=1}^g \cc_i W_i\big)t\biggr],
\end{equation}
with $\alpha_\cc\coloneqq a_\cc\exp[\cc^T {\bf D}]$. 

For tropical curves associated with banana graphs, in \Cref{sec:Hirota_of_graph} we explicitly characterize the Hirota variety, that is we verify under which conditions $\tau_{\mathcal{D}}(x,y,t)$ is a \textit{KP tropical $\tau$-function}, i.e. $u(x,y,t) = 2\partial_x^2 \log(\tau_{\mathcal{D}}(x,y,t))$ is a KP multi-soliton solution.

\subsection{Degenerating hyperelliptic curves}\label{sec:Riemannmatrix}

In this section, we analytically compute both the tropical Riemann matrix $Q$ and the matrix $R =\lim\limits_{\varepsilon \to 0 } R_\varepsilon$ as in \eqref{eq:Repsilon} for the degeneration of hyperelliptic curves that give rise to the banana graph. 
This is one of the ingredients in the explicit parametrization of the Hirota variety  in \Cref{sec:parametrization}. We also define the tropical Abelian map which will be used in \Cref{sec:Hirota_divisors} to parameterize the Hirota variety in terms of divisors. An analytic point of view on periods of tropical curves and their KP solutions is also presented in~\cite{ichikawa2023periods}.

Fix pairwise distinct complex parameters $\kappa_1,\dots,\kappa_{g+1}$ and consider the family
\begin{equation*}
X_\varepsilon\,:\,y^2\,=\,\prod_{i=1}^{g+1}(x-\kappa_i-\varepsilon)(x-\kappa_i+\varepsilon),\qquad \varepsilon>0.
\end{equation*}
As usual, we choose  the standard symplectic homology basis $\{a_i(\varepsilon), b_i(\varepsilon)\}$ on $X_\varepsilon$ for $i=1,\dots,g$: the cycle $a_i(\varepsilon)$ is a loop encircling the branch cut between $\kappa_{i+1}-\varepsilon$ and $\kappa_{i+1}+\varepsilon$ on the upper sheet of the two-sheeted cover of $\mathbb{P}^1_{x}$ (i.e., a small circle around $\kappa_{i+1}$ on $X_\varepsilon$), and the cycle $b_i(\varepsilon)$ is the cycle representing $e_1-e_{i+1}$, obtained by lifting the segment in the $x$–line from $\kappa_1+\varepsilon$ to $\kappa_{i+1}-\varepsilon$ and gluing the two sheets. Let $\omega_1(\varepsilon),\dots,\omega_g(\varepsilon)$ be the unique holomorphic differentials normalized by
\[
\int_{a_i(\varepsilon)}\omega_j(\varepsilon)\,=\,2\pi i \delta_{ij},\qquad i,j=1,\dots,g.
\]
The period matrix $Q(\varepsilon)=(Q_{ij}(\varepsilon))$ with
$Q_{ij}(\varepsilon)\,=\, \int_{b_i(\varepsilon)}\omega_j(\varepsilon)$ is a Riemann matrix, which is complex symmetric with negative definite real part.
As $\varepsilon\to0^+$, the pairs $\kappa_i\pm\varepsilon$ collide and the central fiber $X=X_0$ splits as
\begin{equation}\label{eq:hyperelliptic_curve}
    X \,=\, X_+ \cup X_- , \qquad X_\pm:\ y\,=\,\pm h(x),\quad h(x)\,=\,\prod_{i=1}^{g+1}(x-\kappa_i),
\end{equation}
intersecting transversely at the nodes $n_i=(\kappa_i,0)$, $i=1,\dots,g+1$.  
The normalization map $\nu:\tilde X = X_+ \sqcup X_- \longrightarrow X$ satisfies $\nu^{-1}(n_i)=\{q_i^+,q_i^-\}$ with $q_i^\pm=(\kappa_i,0)\in X_\pm$.  
The dualizing sheaf $\omega_X$ consists of pairs of meromorphic differentials on $X_\pm\simeq\mathbb P^1$ with simple poles at the preimages of the nodes and opposite residues following the residue theorem \cite[Equation (1.4)]{Artamkin2004}. In the limit $\varepsilon\to 0^+$, one obtains the following canonical basis of differentials 
$\omega_1,\dots,\omega_g\in H^0(X,\omega_X)$ with
\begin{equation}\label{eq:basis_holom_diff}
\omega_i\,=\,
\begin{cases}
\dfrac{dx}{x-\kappa_{i+1}}-\dfrac{dx}{x-\kappa_1}, & \text{on }X_+,\smallskip\\
-\Big(\dfrac{dx}{x-\kappa_{i+1}}-\dfrac{dx}{x-\kappa_1}\Big), & \text{on }X_-,
\end{cases}
\qquad i=1,\dots,g,
\end{equation}
so each $\omega_i$ has residues $+1$ at $q_{i+1}^\pm$ and $-1$ at $q_1^\pm$. 
To obtain the limit of the period matrix we replace $\omega_j(\varepsilon)$ by the limiting forms $\omega_j$ and compute the $\varepsilon$–asymptotics explicitly. So we get for the diagonal entries
\[
\begin{aligned}
Q_{ii}(\varepsilon)
&\,=\,2\!\int_{\kappa_1+\varepsilon}^{\kappa_{i+1}-\varepsilon}\Big(\frac{1}{x-\kappa_{i+1}}-\frac{1}{x-\kappa_1}\Big)\,dx\\
&\,=\,2\log(\varepsilon^2)-2\log
(\kappa_{i+1}-\kappa_1)^{2}+O(\varepsilon),
\end{aligned}
\]
and for the off–diagonal entries
\[
\begin{aligned}
Q_{ij}(\varepsilon)
&\,=\,2\!\int_{\kappa_1+\varepsilon}^{\kappa_{i+1}-\varepsilon}\Big(\frac{1}{x-\kappa_{j+1}}-\frac{1}{x-\kappa_1}\Big)\,dx\\
&\,=\,\log(\varepsilon^2)
+\log\!\left(\frac{\kappa_{i+1}-\kappa_{j+1}}{(\kappa_{i+1}-\kappa_1)(\kappa_{j+1}-\kappa_1)}\right)^2
+O(\varepsilon).
\end{aligned}
\]
Collecting the coefficients of $\log(\varepsilon^2)$ gives the tropical Riemann matrix $ Q$ in the homology basis: $Q_{ii}=2$ and $Q_{ij}=1$ with $i\neq j$. The constant terms define the limit Riemann matrix $R=(R_{ij})$ with
\begin{equation}\label{eq:entries_limitR}
R_{ii}\,=\,2\log(\kappa_{i+1}-\kappa_1)^{-2},\qquad
R_{ij}\,=\,2\log\!\left(\frac{\kappa_{i+1}-\kappa_{j+1}}{(\kappa_{i+1}-\kappa_1)(\kappa_{j+1}-\kappa_1)}\right)^2
\quad(i\neq j).
\end{equation}
Thus,
\[
Q(\varepsilon)\,=\,\log(\varepsilon^2)\,Q+R+O(\varepsilon)\qquad(\varepsilon\to0^+).
\]
Note that the entries of the limiting matrices $Q$ and $R$ depend on the choice of homology basis encoded in $B$. However, as pointed out in~\Cref{sec:tropicalJacobians}, a different basis choice gives another matrix related by a $\text{GL}_g(\ZZ)$-action by conjugation.

 Let $p_0^\varepsilon$ be point at infinity on $X_\varepsilon$. With this base point, the Abel-Jacobi image of the divisor $p_1+\dots+p_g$ is defined in the Jacobian of $X_\varepsilon$ is 
\[
\mathcal A_i^\varepsilon(p_1,\dots,p_g)\ = \sum_{\ell=1}^g\int_{p_0^\varepsilon}^{\,p_\ell(\varepsilon)}\omega_i(\varepsilon)
\qquad i=1,\dots,g.
\]
At the limit $X_0$, the base point $p_0^\varepsilon=\infty$ splits to the two points at infinity $\infty_+\in X_+$ and $\infty_-\in X_-$. We view the limit of the Abel-Jacobi map as follows \cite{AbendaGrinevich2025DN}. Assuming that $p_1,\dots,p_k\in X_+$ and $p_{k+1},\dots,p_g\in X_-$ are the limits of $p_\ell^\varepsilon$,
\begin{equation}\label{eq:Riemann_constant}
\mathcal{A}_i^0(p_1,\dots,p_g)=\lim_{\varepsilon\to0}\ \mathcal A_i^\varepsilon(p_1,\dots,p_g)
\, =\, \sum_{\ell=1}^{k}\int_{\infty_+}^{p_\ell}\omega_{i_{|_{X_+}}}\;+\;\sum_{\ell=k+1}^{g}\int_{\infty_-}^{p_\ell}\omega_{i_{|_{X_-}}}
\end{equation}
We adopt the notation in \cite[Section 6]{AgoFevManStu} to define the following polynomials 
\begin{equation}\label{eq:polynomialsPandQ}
    P(z)\,\coloneqq\, \prod_{j=1}^k(z-p_j) \quad \text{and} \quad Q_1\coloneqq1,\,Q_2\coloneqq\frac{1}{z-p_{k+1}},\, \dots,\,Q_{n-k}\coloneqq\frac{1}{z-p_{g}} .
\end{equation}
Then the limit Abel-Jacobi image is:
\begin{equation}\label{eq:AbelJacobi}
\mathcal{A}_i^0(p_1,\dots,p_g)=\log \left( \frac{P(\k_{i+1}) \prod_{l=1}^{n-k} Q_l(\k_{i+1})}{P(\k_1)\prod_{l=1}^{n-k} Q_l(\k_1)}    \right)
\end{equation}
We close this section with a brief remark on the vector of Riemann constants $\mathbf{K}$ as in \Cref{eq:Dparameter}. It is the additive shift that appears in the finite-gap solutions for the KP equation. This vector of $X_\varepsilon$ can be computed with the following formula
\begin{equation}\label{eq:riemannconstant}
    \mathbf{K}_j(\varepsilon)
\,=\,\frac{2\pi i +Q_{jj}(\varepsilon)}{2}
-\frac{1}{2\pi i}\sum_{\ell\neq j}\left(\int_{a_\ell(\varepsilon)}\omega_\ell(\varepsilon)(p)\;
\int_{p_0^\varepsilon}^{p}\omega_j(\varepsilon)\right).
\end{equation}

The analysis of its degeneration as $\varepsilon$ goes to $0$ will be carried out in future work.

\subsection{Periods of meromorphic differentials in the tropical limit}\label{sec:trop_UVW}

We study the vectors $\UU, \VV, \WW$ as the limits of the corresponding vectors $\UU_\varepsilon, \VV_\varepsilon, \WW_\varepsilon$ associated to the degenerating family $X_\varepsilon$ of hyperelliptic curves as $\varepsilon \to 0$. In \Cref{sec:Hirota_of_graph}, we investigate the role of $\UU, \VV, \WW$ in the parametrization of the Hirota variety. 

The vectors $\UU_\varepsilon, \VV_\varepsilon, \WW_\varepsilon$ are defined in terms of period integrals of meromorphic differentials with prescribed poles at the essential singularity $p_0 = \infty$, which is fixed throughout the family. For the classical geometry our main references are \cite{dubrovin, Krich1977}.

For $i\ge 1$ we let $\Omega^{i}_\varepsilon$ to be the normalized meromorphic differential with a pole of order $i+1$ at the point $p_0$ at infinity which is the essential singularity of the KP wave function.
On the hyperelliptic curve $X_\varepsilon$, the KP parameters are obtained as 
\[
\UU_\varepsilon = \left(\oint_{\beta_1^\varepsilon} \Omega^1_\varepsilon, \dots, \oint_{\beta_g^\varepsilon} \Omega^1_\varepsilon\right),
\quad
\VV_\varepsilon = \left(\oint_{\beta_1^\varepsilon} \Omega^2_\varepsilon, \dots, \oint_{\beta_g^\varepsilon} \Omega^2_\varepsilon\right),
\quad
\WW_\varepsilon = \left(\oint_{\beta_1^\varepsilon} \Omega^3_\varepsilon, \dots, \oint_{\beta_g^\varepsilon} \Omega^3_\varepsilon\right).
\]
From \cite[Lemma 2.1.2]{dubrovin}, the periods of these differentials are expressed using the local expansion of holomorphic differentials near $p_0$. Let $\omega_j^\varepsilon = f_j^\varepsilon(\zeta)\,d\zeta$ be the expansion in the local coordinate $\zeta$ centered at $p_0$ with $\zeta(p_0) = 0$. Then,
\[
\oint_{\beta_j^\varepsilon} \Omega^\ell_\varepsilon \,=\, \frac{1}{\ell!} \left. \frac{d^{\ell-1}}{d\zeta^{\ell-1}} f_j^\varepsilon(\zeta) \right|_{\zeta = 0}.
\]
To compute the limits of $\UU_\varepsilon, \VV_\varepsilon, \WW_\varepsilon$, we take $\varepsilon \to 0$ and evaluate
\[
\lim_{\varepsilon \to 0} \oint_{\beta_j^\varepsilon} \Omega^\ell_\varepsilon = \frac{1}{\ell!} \left. \frac{d^{\ell-1}}{d\zeta^{\ell-1}} f_j^0(\zeta) \right|_{\zeta = 0},
\]
where $f_j^0(\zeta)$ is the limit of $f_j^\varepsilon(\zeta)$, obtained from the basis of holomorphic differentials on the central fiber. For this, we use the basis in \eqref{eq:basis_holom_diff}. To ensure compatibility with the local coordinate centered at $p_0 = \infty$, we set $\zeta = 1/x$, so that $dx = -\zeta^{-2} d\zeta$. Depending on whether $p_0$ lies on $X_+$ or $X_-$, we have 
\[
f^{0}_j (\zeta) \,=\, \left\{ 
\begin{array}{ll}
\displaystyle \frac{\k_1-\k_{j+1}}{(\zeta\k_{j+1}-1)(\zeta\k_1-1)}, & \text{ if }\,\, p_0\in X_+,\\
\displaystyle \frac{\k_{j+1}-\k_1}{(\zeta\k_{j+1}-1)(\zeta\k_1-1)}, & \text{ if }\,\, p_0\in X_-.
\end{array}
\right.
\]
Therefore, the derivatives yield
\begin{equation}\label{eq:UVW}
(\UU_j, \VV_j,\WW_j ) \,=\, \left\{ \begin{array}{ll}
(\k_1-\k_{j+1},\,\k_1^2-\k_{j+1}^2,\,\k_1^3-\k_{j+1}^3), & \text{ if }\,\, P_0\in X_+,\smallskip\\
(\k_{j+1}-\k_{1},\,\k_{j+1}^2-\k_{1}^2,\,\k_{j+1}^3-\k_{1}^3), & \text{ if }\,\, P_0\in X_-.\\
\end{array}
\right. 
\end{equation}
These vectors can be computed via the matrix $B$ representing the homology basis of the graph as follows.
Moreover, they fulfill the KP dispersion relations.

\begin{lemma}\label{lem:UVW}
Let $B$ be the matrix as in~\Cref{sec3} representing the homology basis of the banana graph $\Gamma$, and let $K$ be the $n\times 3$ matrix whose entries are $K_{ij}=\k_i^j$ for $i=1,\dots,n$ and $j=1,\dots,3$. Then the vectors $\UU,\VV,\WW$ in \eqref{eq:UVW} are given by
\[
\displaystyle \begin{bmatrix}
        \UU \,\, \VV \,\, \WW 
    \end{bmatrix}\, =\,
\left\{    
\begin{array}{ll}
 \displaystyle \quad B \, \cdot K  & \text{ if } \,\, p_0\in X_+,\\
  \displaystyle  - B \, \cdot K  & \text{ if } \,\, p_0\in X_-,
\end{array}
\right.
\]
and they satisfy satisfy the KP dispersion relation $U_i^4+3V_i^2-4U_iW_i=0$ for all $i\in [g]$. 
\end{lemma}
The parametrization of the vectors $\UU,\VV,\WW$ given in~\Cref{lem:UVW} holds for any choice of homology basis represented by the matrix~$B$. Changing basis, in fact, corresponds to left multiplication by a matrix $G\in {\text{GL}_{g}(\ZZ)}$, as in~\Cref{sec:tropicalJacobians}.

\section{
Hirota varieties of banana graphs}\label{sec:Hirota_of_graph}

This section develops the connection between the tropical Jacobian of a banana graph and the Hirota variety describing KP multi-solitons arising from this family of graphs. The Hirota variety ${\cal H}_{\Gamma}$ associated with a tropical curve $\Gamma$ is the union
\begin{equation}\label{eq:Hirota_Gamma}
\bigcup_{\aa\in {V}_Q } {\cal H}_{\aa,\Gamma},
\end{equation}
where $\aa$ is a point in the Voronoi polytope $V_Q$, not necessarily one of its vertices, and ${\cal H}_{\aa,\Gamma}$ is the Zariski closure of the set in $(\CC^*)^{d}\times \mathbb{W}\mathbb{P}^{3g-1}$ given by 
\[
{\cal H}^{\circ}_{\aa,\Gamma} \coloneqq\left\{ \big( (\alpha_\cc)_{\cc\in \DaQ}  , (\UU,\VV,\WW)\bigr)\, : \, \tau_{\DaQ}(x,y,t) \mbox{ is a KP tropical } \tau \mbox{ function of } \Gamma \right\},
\]
where $d$ is the cardinality of $\DaQ$, and $\tau_{\mathcal{D}} (x,y,t)\equiv\tau_{\DaQ}(x,y,t)$ has been defined in \eqref{eq:tau_from_theta}. 
That means that $u(x,y,t)=2\partial_x^2 \log(\tau_{\mathcal{D}}(x,y,t))$ is a solution to the KP equation \eqref{eq:KPequation} if and only if the data $ ((\alpha_\cc)_{\cc\in \DaQ}  , (\UU,\VV,\WW))\in {\cal H}_\Gamma$.

We remark that this notion of Hirota variety differs from the Hirota variety of a Delaunay polytope introduced in \cite{AgoFevManStu}. The special case in which the corresponding polytopes are hypersimplices is addressed in~\Cref{sec:Hirota_of_hypers}.

If $\aa$ is a Voronoi vertex, i.e.,  $\aa \in{\cal V}_Q$, its Delaunay polytope has maximal dimension $g$ equal to the genus of the tropical curve $X$. That justifies the definition of \textit{the main component of the Hirota variety of} ${\cal H}_\Gamma$ to be the closure of the union of the ${\cal H}_{\aa,\Gamma}$, where $\aa$ varies in the set of Voronoi vertices ${\cal V}_Q$. We denote this main component by ${\cal H}_{\Gamma, {\cal V}_Q}$. In what follows we focus on the study of ${\cal H}_{\Gamma, {\cal V}_Q}$ as we expect that it coincides with the variety ${\cal H}_\Gamma$. 

We now focus on the case of $\Gamma$ being the banana graph. For any given equivalence class $[\kk]$ of Voronoi vertices, the corresponding theta functions differ by a multiplicative exponential factor (see \Cref{lem:equivalence_thetafunct}), so ${\cal H}^\circ_{\aa,\Gamma}={\cal H}^\circ_{\aa^{'},\Gamma}$.
Indeed, if $\aa,\aa'\in [\kk]$, then, for any given admissible triple $(\UU,\VV,\WW)$ and for any given $g$-tuple $\bf D$, the KP solutions $u_{\DaQ}(x,y,t)=2\partial_x^2 \log(\tau_{\DaQ}(x,y,t))$ and $u_{{\cal D}_{\aa',Q}}(x,y,t)=2\partial_x^2 \log(\tau_{{\cal D}_{\aa',Q}}(x,y,t))$ coincide.  Therefore, the main component of the Hirota variety is the union
\[
{\cal H}_{\Gamma, {\cal V}_Q} = \bigcup_{k=1}^g
{\cal H}_{\Gamma,[\kk]}, \qquad \text{with } {\cal H}_{\Gamma,[\kk]}={\cal H}_{\Gamma,\aa} \text{ for some } \aa\in [\kk].
\]

Next, we explore the role of the distinguished point $p_0$ in the tropical limit and in the structure of the Hirota variety.
Without loss of generality, let us identify the component $X_+$ with the vertex $v_1$ and $X_-$ with $v_2$. In the tropical limit, the point $p_0$ may end up in $X_-$ or in $X_+$. As $\UU,\VV,\WW$ are the limit of normalized periods of meromorphic differentials with a single pole of convenient order at $p_0$, in the tropical limit as computed in \Cref{sec:trop_UVW}, the Hirota varieties ${\cal H}_{\Gamma,[{\kk}]}$ naturally split into two subvarieties ${\cal H}_{\Gamma,[{\bf k}],v_1}$ and ${\cal H}_{\Gamma,[{\bf k}],v_2}$:
\begin{equation}\label{eq:Hirota_KK}
{\cal H}_{\Gamma,[\kk]} \; = \; {\cal H}_{\Gamma,[\kk],v_1}\, \cup \, {\cal H}_{\Gamma,[\kk], v_2}.
\end{equation}
In the following sections we provide a canonical parametrization of $ {\cal H}_{\Gamma,[\kk],v_i}$ for $i=1,2$ using the explicit computation of the matrix $R$. Then in \Cref{sec:Hirota_divisors} we explain the role of KP divisors in this parametrization. 

\subsection{A parametrization from the tropical data}\label{sec:parametrization}

Recall that the vertices of the Voronoi polytope $V_Q$ are divided into equivalence classes $[{\bf 1}],\dots,[{\bf g}]$ described in~\Cref{thm:vertices_voronoi_Rg}. Let $I_k\coloneqq\{1,2,\dots,k\}.$ For each class, we choose, for convenience, the Voronoi vertex $\aa_{I_k}$, where
\begin{equation}\label{eq:a_minlex}
\aa_{I_k}\, = \,\Bigg( \underbrace{\frac{g+1-k}{g+1},\,\dots,\, \frac{g+1-k}{g+1}}_{ k-1 \text{ times}}, \underbrace{  -\frac{k}{g+1},\, \dots,\,  -\frac{k}{g+1}}_{ g+1-k \text{ times}}\Bigg)\in \RR^g.
\end{equation}
The point $\aa_{I_k}$ is the Voronoi vertex whose associated strongly connected orientation $\iota_{\aa_{I_k}}$ has the first $k$ edges outgoing at $v_1$; equivalently, the Delaunay vertex $\cc =(0,...,0)\in {\cal D}_{\aa_{I_k}}$ corresponds to the minimum lexicographic base $I_k=\{1,\dots,k\} \in {\cal B}_{\aa_{I_k},v_1}$. 

The varieties $\mathcal{H}_{\Gamma,[{\bf k}],v_i}$ parametrize the KP multi-solitons arising from a hyperelliptic curve~\eqref{eq:hyperelliptic_curve}, depending on the choice of component $X_+$ or $X_{-}$ containing the essential singularity $p_0$. Combinatorially, this choice is encoded by selecting the matroid $\mathcal{M}_{\aa,v_i}$ as in \eqref{eq:matroid_v1} and \eqref{eq:matroid_v2}. Therefore, using the correspondence between ${\cal B}_{\aa,v_i}$ and the set of Delaunay vertices $\DaQ$, we label each Delaunay vertex with a corresponding base $I\in {\cal B}_{\aa,v_i}$. We consider the rational map 
\begin{align}\label{eq:psi_Simo}
    \psi_{[{\bf k}],v_1}\, : \, \CC^{2g+1} &\dashrightarrow \KWPB\\
(\underline{\beta},\underline{\kappa}) &\longmapsto ((\alpha_I)_{I\in {\cal B}_{\aa,v_1}},(\UU,\VV,\WW)).\nonumber
\end{align}
with image given by
\begin{equation}\label{eq:image_v1}
\displaystyle
    \alpha_{I} \,=\, \exp\left[ \frac{1}{2} \cc_I^T R \cc_I\right] \cdot \exp[ \cc_I^T {\bf D}] ,
\qquad
\displaystyle \begin{bmatrix}
        \UU \,\, \VV \,\, \WW 
    \end{bmatrix}\, =\, B \, \cdot K,
\end{equation}
where $I\in{\cal B}_{\aa,v_1}$, ${\bf D}=(\log \beta_1,\dots,\log \beta_g)$, and the matrices $B$ and $K$ are as in~\Cref{lem:UVW}. Analogously, the map $\psi_{[{\bf k}],v_2}$ has image 
\textbf{\begin{equation}\label{eq:image_v2}
\displaystyle
    \alpha_{\bar I} \,=\, \exp\left[ \frac{1}{2} \cc_{\bar I}^T R \cc_{\bar I}\right] \cdot \exp[ \cc_{\bar I}^T {\bf D}] ,
\qquad
\displaystyle \begin{bmatrix}
        \UU \,\, \VV \,\, \WW 
    \end{bmatrix}\, =\, -B \, \cdot K,
\end{equation}}
where ${\bar I}=[n]\setminus I\in{\cal B}_{\aa,v_2}$.

\begin{theorem}\label{theo:Hirota_parametrization}
The map $\psi_{{[{\bf k}],v_i}}$ for $i=1,2$ is a birational map onto its image $\mathcal{H}_{\Gamma,[{\bf k}],v_i}$, which is an irreducible algebraic variety of dimension $2g+1$. Moreover, 
\begin{enumerate}
\item The variety $\mathcal{H}_{\Gamma,[{\bf k}],v_1}$ (resp. $\mathcal{H}_{\Gamma,[{\bf k}],v_2}$)  parameterizes the family of KP multi-solitons in $\Gr(k,n)$ (resp. $\Gr(n-k,n)$) arising from the degeneration of the hyperelliptic curve in~\eqref{eq:hyperelliptic_curve}. 
\item For any $(\underline{\beta},\underline{\k})$, the KP solutions $u_{[{\bf k}],v_1} (x,y,t)$ and $u_{[{\bf k}],v_2} (x,y,t)$ given by the image of the map $\psi_{[{\bf k}],v_i}$, for $i=1,2$, are mapped to each other via space--time inversion $(x,y,t) \mapsto (-x,-y,-t)$:
\begin{equation}\label{eq:u_for_Vi}
    u_{[{\bf k}],v_2} (x,y,t) \,=\,u_{[{\bf k}],v_1} (-x,-y,-t).
\end{equation}
\end{enumerate}
\end{theorem}

We explain the details of the construction in \Cref{sec:parametrization}: we fix the vertex $v_1$, and prove \Cref{theo:Hirota_parametrization} in the case of $\psi_{[\kk],v_1}$. The idea of the proof is to verify that the function $\tau_{\cal{D}}(x,y,t)$ in \eqref{eq:tau_from_theta} constructed as the image of $\psi_{[\kk],v_1}(\underline{\beta},\underline{\k})$ is a $\tau$--function for the family of KP multi-solitons associated with $([\kk],v_1)$, and show that such solutions form a subvariety in $\Gr(k,n)$, $n=g+1$. Then the property that the family of KP solitons associated with $v_2$ and $[\kk]$ is a subvariety in the dual Grassmannian $\Gr(n-k,n)$ is a direct consequence of the fact that the matroids ${\cal M}_{\aa,v_1}$, ${\cal M}_{\aa,v_2}$ are dual to each other. 

\begin{proof}[Proof of \Cref{theo:Hirota_parametrization}, Item~2]
We separately prove Item~2, as it immediately follows by inspecting \eqref{eq:image_v1} and \eqref{eq:image_v2}. Indeed, \Cref{lem:UVW} implies that, if a vector $(\UU,\VV,\WW)$ belongs to the image of the map $\psi_{[{\bf k}],v_1}$, then its opposite $(-\UU,-\VV,-\WW)$ belongs to the image of the map $\psi_{[{\bf k}],v_2}$. Finally, for any $(\underline{\beta},\underline{\k})\in \CC^{2g+1}$, $
 \alpha_{\bar I} (\underline{\beta},\underline{\k}) \,=\, \alpha_{I} (\underline{\beta},\underline{\k}), 
 $
with $I\in {\cal B}_{\aa,v_1}$, ${\bar I}\in {\cal B}_{\aa,v_2}$, and ${\bar I}= [n]\setminus I$.
\end{proof}

In the following, we omit the subscript $v_1$ in the notation, unless necessary. We start constructing a $k\times n$ matrix $A$ and computing its maximal minors.

Let $\bar{\aa}=\aa_{I_k} \in [{\bf k}]$ be the Voronoi vertex given in~\eqref{eq:a_minlex}. 
Any basis $J\in {\cal B}_{\bar{\aa}}$ can be obtained from $I_k$ as $J =I_k \setminus \{ i_1,\dots, i_s\} \cup \{j_1,\dots,j_s\}$ where $1\le s\le \min \{k, n-k\}$ and $i_l,j_m\in[n]$. Then the corresponding Delaunay vertex $\cc_J$ satisfies
\[
B^T \cc_J + {\bf s}_{\bar{\aa}}\, = \sum_{i \in I_k\setminus J} \tilde{\ee}_i + \sum_{j\in J\setminus I_k} \tilde{\ee}_j, \qquad \text{and} \qquad \cc_J \,=\, \sum_{l=1}^s \left( {\ee}_{i_l-1} - {\ee}_{j_l-1} \right),
\]
where ${\bf s}_{\bar{\aa}}$ has $1$ in the first $k$ coordinates and $0$ otherwise, and $\tilde{\ee}_i$ and $\ee_i$ denote the $i$-th standard basis vector in $\RR^n$ and $\RR^g$, respectively, with ${\ee}_0$ the null vector. The image of the $\alpha_J$ from  \eqref{eq:image_v1} is explicitly given by 
\begin{equation}\label{eq:alpha_J}
    \alpha_{J}\,=\,\frac{\displaystyle \prod_{1\le l<m\le s}  (\kappa_{i_m} -\kappa_{i_l})^2 (\kappa_{j_m} -\kappa_{j_l})^2}{\displaystyle \prod_{1\le l,m \le s}  (\kappa_{j_m} -\kappa_{i_l})^2} \prod_{l=1}^g\beta_l^{-(B^T\cc_J)_{l+1}}, \qquad J\in {\cal B}_{\bar{\aa}},
\end{equation}
where the numerator is set to $1$ when $s=1$. Let $A$ be the $k\times n$ matrix in reduced row echelon form (RREF) with entries
\begin{equation}\label{eq:matrix_A}
    A_{ij} \,=\, \left\{  \begin{array}{ll} \delta_{ij} &\quad  \text{ if }  i,j\le k,\\
\displaystyle \frac{\beta_{i-1}}{\beta_{j-1} (\k_j-\k_i)^2} \prod_{\substack{l\in I_k\\l\not = i}} \frac{\k_i-\k_l}{\k_j-\k_l} &\quad \text{ if }i \le k \text{ and } k<j\le n,
\end{array}
\right.
\end{equation}
where $\beta_0=1$. The following result shows that the $\alpha_J$ described in \eqref{eq:alpha_J} can be viewed as maximal minors of the matrix $A$ up to a multiplicative factor.

\begin{proposition}\label{prop:minor_iden}
    Let $\alpha_J$, with $J\in {\cal B}_{\bar{\aa}}$ be as in \eqref{eq:alpha_J}, and $A\in \Gr(k,n)$ be the matrix defined in \eqref{eq:matrix_A}. Then, for any $J\in {\cal B}_{\bar{\aa}}$ the $J$-th maximal minor of $A$, denoted $A_J$, satisfies
\begin{equation}\label{eq:minor_identity}
A_J \,=\, \alpha_J \frac{K_{I_k}}{K_J},  
\end{equation}
with $K_{I_k}$ and  $K_{J}$ as in~\eqref{eq:KJ_EJ}.
\end{proposition}

\begin{proof}
It is sufficient to prove  (\ref{eq:minor_identity}) in the case $\beta_j=1$ for $j\in [g]$.
Let $J= I_k\setminus \{ i_1,\dots, i_s\} \cup \{j_1,\dots, j_s\}$ as above, and define $\sigma_J= \sum_{l=1}^s (k-i_l)$. Let us check the case $s=1$:
\[
A_{I_k\setminus \{i\} \cup \{j\}} \, = \, (-1)^{k-i} A_{ij} \, = \, \frac{(-1)^{k-i}}{(\k_j-\k_i)^{2} }\, \prod_{\substack{l\in I_k\\ l\not = i}} \frac{\k_i-\k_l}{\k_j-\k_l} \, = \, \alpha_J  \,\frac{K_{I_k}}{K_{I_k\setminus \{i\} \cup \{j\}}}.
\]
If $s=2$ and $J= I_k\setminus \{ i_1,i_2\} \cup \{j_1,j_2\}$, then
\[
\begin{array}{ll}
A_J \!\!\!\!&= (-1)^{\sigma_J} \left( A_{i_2 \, j_2} A_{i_1 \, j_1} -A_{i_2 \, j_1} A_{i_1 \, j_2}\right)\\
&=\displaystyle (-1)^{\sigma_J} \biggl( \frac{1}{(\k_{j_2}-\k_{i_2})^2(\k_{j_1}-\k_{i_1})^2} \prod_{l\in I_k\setminus\{i_1\}} \frac{(\k_{i_1}-\k_l)}{(\k_{j_1}-\k_l)} \prod_{l\in I_k\setminus\{i_2\}} \frac{(\k_{i_2}-\k_l)}{(\k_{j_2}-\k_l)}  - (j_1 \leftrightarrow  j_2) \biggr),\end{array}
\]
where $(j_1 \leftrightarrow  j_2)$ indicates the first summand where we switch $j_1$ with $j_2$. Thus we have:
\[
\begin{array}{ll}
A_J \!\!\!\!&=\displaystyle \frac{(-1)^{\sigma_J}(\k_{i_2}-\k_{i_1})^2\cdot R_{12}}{(\k_{j_2}-\k_{i_2})^2(\k_{j_1}-\k_{i_1})^2(\k_{j_2}-\k_{i_1})^2(\k_{j_1}-\k_{i_2})^2} \prod_{\substack{i\in I_k\\i\neq i_1,i_2}}  \frac{(\k_{i_1}-\k_i)(\k_{i_2}-\k_i)}{(\k_{j_1}-\k_i)(\k_{j_2}-\k_i)}   \\
&\\
&=\displaystyle  \frac{(-1)^{\sigma_J}(\k_{i_2}-\k_{i_1})^3(\k_{j_2}-\k_{j_1})}{(\k_{j_2}-\k_{i_2})^2(\k_{j_1}-\k_{i_1})^2(\k_{j_2}-\k_{i_1})^2(\k_{j_1}-\k_{i_2})^2} \prod_{\substack{i\in I_k\\i\not = i_1,i_2}}  \frac{(\k_{i_1}-\k_i)(\k_{i_2}-\k_i)}{(\k_{j_1}-\k_i)(\k_{j_2}-\k_i)}\,=\, \displaystyle \alpha_J \frac{K_I}{K_J},
\end{array}
\]
where $R_{12} = (\k_{j_2}-\k_{i_2}) (\k_{j_1}-\k_{i_1})-(\k_{j_2}-\k_{i_1}) (\k_{j_1}-\k_{i_2})$. We now consider the general case.  Define $\nu_s\coloneqq(s-1)(s-2)/2$. Then $A_J = (-1)^{\sigma_J} \sum_{\pi} (-1)^{\varepsilon_\pi} A_{i_1 \, j_{\pi(1)}} \cdots A_{i_s \, j_{\pi(s)}}$, where the sum runs over all permutations of $\{1,\dots,s\}$ and $\varepsilon_\pi$ is the parity of the permutation. It is straightforward to check
\[
\begin{array}{ll}
A_J \!\!\!\!&=\! \displaystyle (-1)^{\sigma_J} \sum_{\pi} (-1)^{\varepsilon_\pi} \prod_{r=1}^s \biggl( (\k_{j_{\pi(r)}} -\k_{i_r})^{-2} \prod_{l\in I_k\setminus\{i_r\}} \frac{(\k_{i_r}-\k_l)}{(\k_{j_{\pi(r)}} -\k_l)} \biggr)\\
&=\!\displaystyle (-1)^{\sigma_J+\nu_s}\! \, \biggl(\prod_{r=1}^s \prod_{l\in I_k\cap J}\frac{(\k_{i_r}-\k_l)}{(\k_{j_{\pi(r)}}-\k_l)} \biggr) \frac{\displaystyle \prod_{1\le l<m\le s} (\k_{i_m}-\k_{i_l})^2}{\displaystyle \prod_{1\le l,m\le s} (\k_{j_{\pi(m)}}-\k_{i_l})^2} \, \!\!\sum_{\pi} (-1)^{\varepsilon_\pi}
\prod_{r=1}^s \prod_{\substack{l\in I_k\setminus J\\l\not = i_r}} (\k_{j_{\pi(r)}}- \k_l)
\\
&=\!\displaystyle  (-1)^{\sigma_J+\nu_s} \, \biggl(\prod_{r=1}^s \prod_{l\in I_k\cap J}\frac{(\k_{i_r}-\k_l)}{(\k_{j_{\pi(r)}}-\k_l)} \biggr) \frac{\displaystyle \prod_{1\le l<m\le s} (\k_{i_m}-\k_{i_l})^2}{\displaystyle \prod_{1\le l,m\le s} (\k_{j_{\pi(m)}}-\k_{i_l})^2} \, K_{I_k\setminus J}\, K_{J\setminus I_k},
\end{array}
\]
where the last equality uses~\Cref{lem:pluecker_vandermonde}. Then \eqref{eq:minor_identity} follows immediately.
\end{proof}

\begin{lemma}\label{lem:pluecker_vandermonde}
Let $2\le s \le k$ be fixed. Let 
$I_s=\{i_1<\cdots < i_s\}$, $J_s=\{j_1<\cdots < j_s\}$, with $1\le i_1< \cdots < i_s < j_1 < \cdots <j_s\le n$. Define $\nu_s=(s-1)(s-2)/2$.
Then, the identity holds
\begin{equation}\label{eq:pluecker_vandermonde}
    \sum_{\pi_s} (-1)^{\varepsilon_{\pi_s}+\nu_s}
\prod_{r=1}^s \prod_{\substack{l\in I\setminus J\\l\not = i_r}} (\k_{j_{\pi_s(r)}}- \k_l)
\, =\, K_{I_s} \, K_{J_s},
\end{equation}
where $\pi_s$ range over all permutations of $\{1,\dots,s\}$, and $\varepsilon_{\pi_s}$ denote their parity.
\end{lemma}

\begin{proof}
    The proof consists of verifying the equivalence of the above identity with the usual Pl\"ucker relations for the Vandermonde submatrix of size $s\times n$ with entry $(i,j)$ given by $\kappa^i_j$. The proof proceeds by induction on $s$. In particular, if 
    $s=2$, and $I=\{i_1<i_2\}$ and $J=\{j_1<j_2\}$, the left hand side of \eqref{eq:pluecker_vandermonde} is
    \[
    \begin{array}{l}
(\k_{j_1}-\k_{i_1}) (\k_{j_2}-\k_{i_2})-(\k_{j_1}-\k_{i_2})(\k_{j_2}-\k_{i_1}) =\\
\\
 \displaystyle \quad \, = \,K_{(I_2\setminus \{i_2\} )\cup \{j_1 \}} K_{(J_2\setminus \{j_1\}) \cup\{i_2\}}
    -K_{(I_2\setminus \{i_2\}) \cup \{j_2 \}} K_{(J_2\setminus \{j_2\} )\cup\{i_2\}} = K_{I_2} \, K_{J_2}.
    \end{array}\]
So the base case holds. Assume the identity for any $2\le i\le s-1$ and fix ordered sets 
$I=\{i_1<\cdots<i_s\}$, $J=\{j_1<\cdots<j_s\}$. Then for $i=s$, we get
\[
\begin{array}{l}
 \displaystyle\sum_{\pi_s} (-1)^{\varepsilon_{\pi_s}+\nu_s}
\prod_{r=1}^s \prod_{\substack{l\in I_s\\l\not = i_r}} (\k_{j_{\pi_s(r)}}- \k_l)
  \\
\displaystyle\quad \, = \, \sum_{m=1}^s \sum_{\substack{\pi_s \\ \pi_s(s)=m}}(-1)^{\varepsilon_{\pi_s}+\nu_s}\prod_{p\in[s]\setminus m} (\k_{j_p}-\k_{i_s})\, \prod_{l=1}^{s-1} (\k_{j_m}-\k_{i_l})\,
\prod_{r=1}^{s-1} \, \prod_{\substack{l=1 \\ l\not =\pi_s(r)}}^{s-1} (\k_{j_{\pi_s(r)}} - \k_{i_l}) 
\\
\displaystyle \quad \, = \,
\sum_{m=1}^s (-1)^m \prod_{p\in[s]\setminus m} (\k_{j_p}-\k_{i_s})\, \prod_{l=1}^{s-1} (\k_{j_m}-\k_{i_l})\,\sum_{\pi_{s-1}} (-1)^{\varepsilon_{\pi_{s-1}}+\nu_{s-1}}
\prod_{r=1}^{s-1} \, \prod_{\substack{l=1 \\ l\not =\pi_s(r)}}^{s-1} (\k_{j_{\pi_s(r)}} - \k_{i_l})
\\
\displaystyle \quad \, = \, \sum_{m=1}^s (-1)^m \prod_{p\in[s]\setminus m} (\k_{j_p}-\k_{i_s})\, \prod_{l=1}^{s-1} (\k_{j_m}-\k_{i_l})\, K_{I_s \setminus \{i_s\}} \, K_{J_s \setminus \{j_m\}} \,
\\
\displaystyle \quad \,= \,\sum_{m=1}^s (-1)^m \, K_{(I_s \setminus \{i_s\}) \cup \{j_m\}} \, K_{(J_s \setminus \{j_m\}) \cup \{ i_s\}} \, = \, K_{I_s}\, K_{J_s}.
\end{array}
\]    
This computation concludes the proof.
\end{proof}

One immediate consequence of \Cref{prop:minor_iden}  is that the matroid associated with the matrix $A\in \Gr(k,n)$ coincides with the matroid induced by the Delaunay polytope. 
\begin{corollary}\label{prop:matroid_A}
    The matroid of the $k\times n$ matrix $A$ is precisely ${\cal M}_{\bar{\aa},v_1}$.
\end{corollary}

We are now ready to prove \Cref{theo:Hirota_parametrization}.

\begin{proof}[Proof of \Cref{theo:Hirota_parametrization}]
    Let $S\subset(\CC)^{2g+1}$ be the locus where at least two of $\k_i$'s coincide. We want to show that 
\begin{equation}\label{eq:tau_1}
    \tau_{{\cal D}}(x,y,t)\,=\,\theta_{\cal D}(\UU x+\VV y+\WW t+{\bf D}) \,=\, \sum_{J\in  {\cal B}_{\bar{\aa}}} \alpha_J \exp[c_J^T\cdot (\UU x+\VV y+\WW t)],
\end{equation}
is a KP $\tau$-function, 
where ${\cal D} = {\cal D}_{\bar{\aa},Q}$ and $((\alpha_J)_{J\in {\cal B}_{\bar{\aa}}}, \UU,\VV,\WW)$ is the image of a point $(\underline{\beta},\underline{\kappa})\in \CC^{2g+1}\setminus S$ via $\psi_{[{\bf k}]}$. 

First, we focus on the image of the vectors $(\UU,\VV,\WW)$. For a basis $J\in {\cal B}_{[{\bf k}]}$, we have
$$\cc_J^T \cdot (\UU x+\VV y+\WW t) \,=\, (B^T \cc_J)^T\cdot K\cdot  {\bf x}\
    \, \,=\, \, \sum_{j=1}^n (B^T \cc_J)_j (\k_jx+ \k_j^2 y + \k_j^3 t),$$
    where $\bf x$ is the column vector with entries $x,y,t$. Passing to the exponential we get 
    \[
\exp[\cc_J^T\cdot (\UU x+\VV y+\WW t)] \, = \, \prod_{j=1}^n \exp[\k_j x+ \k_j^2 y +\k_j^3 t]^{(B^T\cc_J)_j} \,=\, \frac{E_J(x,y,t)}{E_{I_k}(x,y,t)},
    \]
    where in the last equality we have used that 
    $$\exp[{\bf s}_{\bar{\aa}}^T\cdot  (\UU x+\VV y+\WW t)]\, =\, \prod_{i=1}^k \exp[\k_ix + \k_i^2 y+  \k_i^3 t] \, =\,  E_{I_k}(x,y,t),$$ and  $$\exp[(B^T \cc_J +{\bf s}_{\bar{\aa}})^T (\UU x+\VV y+\WW t)] \,=\,\prod_{j\in J}\exp(\k_jx + \k_j^2 y+  \k_j^3 t)  \,=\,  E_J(x,y,t).$$
    Then function in~\eqref{eq:tau_1} is a KP $\tau$--function for the matrix $A$ in \eqref{eq:matrix_A}:
    \[
    \tau_{\cal D} (x,y,t) \,= \,\frac{1}{K_{I_k} E_{I_k}(x,y,t)} \sum_{J\in {\cal B}_{\bar{\aa}}} A_J K_J E_J(x,y,t) 
  \, = \,  \frac{1}{K_{I_k} E_{I_k}(x,y,t)}\tau_A (x,y,t),
    \]
    where in the first equality we have also used~\Cref{prop:minor_iden} and $\tau_A(x,y,t)$ denotes the $\tau$-function associated to the matrix $A$ defined in \eqref{eq:tauGrass}.
    Therefore, the image of $\CC^{2g+1}$ through the map $\psi_{[\kk]}$ gives a KP multi-soliton solution in $\Gr(k,n)$.

   The map $\psi_{[\kk]}$ is invertible outside the locus $S\subset\CC^{2g+1}$ where the image vanishes. Indeed, we have
    \[
    \k_{j+1} = \frac{V_j - U_j^2}{2 U_j}, \quad\quad \k_1 =\frac{V_j +U_j^2}{2U_j}, \quad\quad j\in [g].
    \]
    The $\beta_j$ may be reconstructed recursively:
    \[
    \beta_j \,=\, \frac{1}{(\k_{j+1}-\k_1)^{2} \left(\alpha_{{I_k} 
\setminus \{ 1\} \cup \{ j+1\} }\right)} \qquad \text{and} \qquad \beta_{i} \,=\, (\k_n-\k_{i+1})^2 \beta_g \alpha_{I_k\setminus \{ i+1\} \cup \{ n\} }
    \]
    for $j=k,\dots, g,$ and $i=1,\dots,k.$ Hence, we can conclude that the map $\psi_{[\kk]}$ is birational. This implies that the closure of the image is irreducible and has dimension $2g+1$. Finally, note that the image of the map depends only on the equivalence class of the chosen Voronoi vertex. This follows from~\Cref{lem:equivalence_thetafunct} as theta functions of Delaunay polytopes arising from equivalent Voronoi vertices just differ by an exponential factor and therefore lead to the same KP solution, see~\eqref{eq:u_tau}.
\end{proof}

We conclude with an example.
\begin{example}[Genus $3$]
  Consider the banana graph of genus $3$. Let $\bar{\aa} = (\frac{1}{2},-\frac{1}{2},-\frac{1}{2})\in [\textcolor{mygreen}{{\bf 2}}]$ as in~\eqref{eq:a_minlex}. The corresponding Delaunay set is
    $${\cal D}_{\bar{\aa},Q}\,=\,\{(0,0,0)_{12},(0,-1,0)_{23},(0,0,-1)_{24},(1,-1,-1)_{34},(1,-1,0)_{13},(1,0,-1)_{14}\},$$
    where each vertex is indexed by the associated matroid basis of $\mathcal{M}_{{\bar{\aa}},v_1}$ as in~\eqref{eq:matroid_v1}. The degenerate theta function is 
    \begin{align*}
   \theta_{\mathcal{D}_{\bar{\aa},Q}}\,(\zz) \, = \, &\alpha_{000}+\aa_{0-10}\exp[-z_2]+\aa_{00-1}\exp[-z_3]+\aa_{1-1-1}\exp[z_1-z_2-z_3]+\\
   &\aa_{1-10}\exp[z_1-z_2]+\aa_{10-1}\exp[z_1-z_3].
    \end{align*}
    The limiting Riemann matrix $R$ is the symmetric matrix
    \begin{equation*}
    R\, = \, \begin{bmatrix}
        -2\log(\k_2-\k_1)^{2} & \log\displaystyle{\frac{(\k_3-\k_2)^2}{(\k_2-\k_1)^2\cdot (\k_3-\k_1)^2}} & \log\displaystyle{\frac{(\k_4-\k_2)^2}{(\k_4-\k_1)^2\cdot(\k_2-\k_1)^2}}\\
        * & -2\log(\k_3-\k_1)^{2} & \log\displaystyle{\frac{(\k_4-\k_3)^2}{(\k_4-\k_1)^2\cdot (\k_3-\k_1)^2}}\\
        * & * &-2\log(\k_4-\k_1)^{2}
    \end{bmatrix} 
    \end{equation*}
    From~\eqref{eq:image_v1}, the image of the vectors $(\UU,\VV,\WW)$ via the map $\psi_{[{\bf 2}]}$ in \eqref{eq:psi_Simo} is   
    \begin{equation*}
    \begin{matrix}
U_1=\k_1-\k_2, & V_1=\k_1^2-\k_2^2, & W_1 = \k_1^3-\k_2^3, \smallskip\\
U_2=\k_1-\k_3, & V_2=\k_1^2-\k_3^2, & W_2 =\k_1^3-\k_3^3, \smallskip\\
U_3=\k_1-\k_4, & V_3=\k_1^2-\k_4^2, & W_3 = \k_1^3-\k_4^3.
    \end{matrix}
    \end{equation*}
Recall that ${\bf s}_{\bar{\aa}}=(1,1,0,0)$. Evaluating $\theta_{\mathcal{D}_{\bar{\aa},Q}}(\UU x + \VV y + \WW t+{\bf D})$ at a point in the image of  $\psi_{[{\bf 2}]}$, we get 
\begin{align*}
\tau_{\mathcal{D}_{\bar{\aa},Q}}(x,y,t) =  \frac{1}{K_{12}E_1E_2} \biggl( 
& K_{12}E_1E_2 + \frac{K_{12}}{K_{13}^2\beta_2}E_2E_3 +\frac{K_{12}}{K_{14}^2\beta_3}E_2E_4 \\ &+\frac{K_{12}^3K_{34}^2\beta_1}{K_{13}^2K_{23}^2K_{14}^2K_{24}^2\beta_2\beta_3}E_3E_4
+\frac{K_{12}\beta_1}{K_{23}^2\beta_2}E_1E_3+ \frac{K_{12}\beta_1}{K_{24}^2\beta_3}E_1E_4\biggr),
\end{align*}
and the expression in the parentheses is precisely the $\tau$-function of the $(2,4)$-soliton associated with the matrix
\begin{equation*}
A \, = \, \begin{bmatrix}
    1 & 0 & \displaystyle{\frac{1}{\beta_2(\k_3-k_1)^2}}\frac{\k_1-\k_2}{\k_3-\k_2} & \displaystyle{\frac{1}{\beta_3(\k_4-k_1)^2}}\frac{\k_1-\k_2}{\k_4-\k_2}\smallskip\\
     0 & 1 & \displaystyle{\frac{\beta_1}{\beta_2(\k_3-k_2)^2}}\frac{\k_2-\k_1}{\k_3-\k_1} & \displaystyle\frac{1}{\beta_3(\k_4-k_2)^2}\frac{\k_2-\k_1}{\k_4-\k_1}\smallskip\\
\end{bmatrix} \in \Gr(2,4).
\end{equation*}
\end{example}

\subsection{Hirota parametrization from Krichever's method}

In this section, we provide an alternative parametrization of the variety ${\mathcal H}_{\Gamma,[{\bf k}],v_1}$ using the construction of KP soliton solutions given in~\cite{Abenda2017}. The family of KP soliton solutions given in~\cite{Abenda2017} follows from Krichever's method for spectral problems for degenerate finite-gap solutions such as multi-solitons. For more details, we refer to ~\cite{Abenda2017,krichever1986spectral,krichever2002toda}. Here, we introduce such a family of solitons and compare their $\tau$-function with that induced by the parametrization in~\eqref{eq:psi_Simo}. This new parametrization is also used in the next section to express the Hirota variety as the image of the KP divisors.

In \cite{Abenda2017}, it is proven that, for any given $k\in [g]$, the map $(\lambda_1,\dots, \lambda_g, \k_1,\dots, \k_n)$ to $\Gr(k,n)$, given by 
    \begin{equation}\label{eq:matrix_A_JGP}
        \tilde A \, = \, \begin{pmatrix}
            1 & 1 & \dots & 1\\
            \kappa_1 & \kappa_2 & \dots & \kappa_n \\
            \vdots & \vdots & \dots & \vdots\\
            \kappa_1^{k-1} & \kappa_2^{k-1} & \dots & \kappa_n^{k-1} \\
        \end{pmatrix}\, \cdot \,
     \text{diag}(1,\lambda_1,\dots,\lambda_{g}),
    \end{equation}
    characterizes the family of $(k,n)$-soliton solutions on $X$ with $\tau$-function $\tau_{\tilde A} (x,y,t)$ as in \eqref{eq:tauGrass}.
The following proposition shows that the matroid associated with 
$A$ in \eqref{eq:matrix_A} and with $\tilde A$ in \eqref{eq:matrix_A_JGP} are the same.  

\begin{proposition}\label{prop:matroid_tilde_A}
The matroid of the $k\times n$ matrix $\tilde A$ is ${\cal M}_{\aa, v_1}$, the matroid associated with each Voronoi vertex $\aa\in [{\bf k}]$ at the vertex $v_1$ of $\Gamma$.
\end{proposition}

\begin{proof}
The proof follows by a direct computation of the maximal minors of $\tilde{A}.$
\end{proof}
The next proposition provides the explicit relation between the parameters $\beta_1,\dots,\beta_g$ used in the parametrization of the Hirota variety in \Cref{theo:Hirota_parametrization} and the parameters $\lambda_1,\dots,\lambda_g$ which appear in \eqref{eq:matrix_A_JGP}.

\begin{proposition}\label{lem:beta_lambda}
Let $\k_1,\dots,\k_n$ be given. Then the matrices $A$ and $\tilde A$ as in \eqref{eq:matrix_A} and \eqref{eq:matrix_A_JGP} represent the same point in $\Gr(k,n)$ if and only if 
\begin{equation}\label{eq:beta_lambda}
\beta_j \, = \, \lambda_j^{-1} \exp \left[ \frac{k}{2} R_{jj} - \sum_{l=1}^{k-1} R_{jl} \right],
\quad\quad j\in [g].
\end{equation}
\end{proposition}
    
\begin{proof}
    The statement follows imposing the equality between the entries of ${\tilde A}_{RREF}$, the reduced row echelon form  of the matrix $\tilde A$ in \eqref{eq:matrix_A_JGP} ($i \in [k]$, $j\in [n]$),
       \[
    ({\tilde A}_{RREF})_{i \, j} \, = \, (-1)^{k+i} \frac{{\tilde A}_{I_k\setminus \{i\} \cup \{j\}}}{A_I} \, =\, (-1)^{k+i}\frac{\lambda_{j-1}}{\lambda_{i-1}}  \frac{K_{I\setminus \{i\} \cup \{j\}}}{K_I},
    \]
    and the entries of the matrix $A$ in \eqref{eq:matrix_A}
    \begin{align*}
         A_{ij} \, &=  \frac{(-1)^{k+i} K_I}{K_{I\setminus \{i\} \cup \{j\}}}\alpha_{I\setminus \{i\}\cup \{j\}}\\ &=
  (-1)^{k+i}\frac{\beta_{i-1}}{\beta_{j-1}} \exp\left[ \frac{1}{2} (R_{i-1,i-1} + R_{j-1,j-1})-R_{i-1,j-1} \right]\frac{K_I}{K_{I\setminus \{i\} \cup \{j\}}},
    \end{align*}
where we use the notation $\lambda_0=1$, $\beta_0=1$, $R_{0,i}=0$, and $I\coloneqq I_k = \{1,\dots,k\}$.
Then \eqref{eq:beta_lambda} follows using the formula in \Cref{lem:K(R)} to express the minors of the Vandermonde submatrix as a function of the elements of the Riemann matrix $R$.
\end{proof}

\begin{lemma}\label{lem:K(R)}
    Let $K$ be the Vandermonde submatrix of size $k\times n$ with entries $K_{ij} = \k_j^{i-1}$, and $K_J$ denote the maximal minors with column indices $J=\{ 1\le j_1 < \cdots < j_k \le n \}$. Then
\begin{equation}\label{eq:squared_Vandermonde_minors}
    K_J^2 \,=\, \exp\left[ -\frac{k-1}{2} \sum_{l=1}^{k} R_{j_l-1,j_l-1} +\sum_{1\le l_1 < l_2 \le k} R_{j_{l_1}-1,j_{l_2}-1} \right].
\end{equation}
\end{lemma}
\begin{proof}
     Note that for $I=\{1,j+1\}$ and $\tilde{I} =\{i+1,j+1\}$ we get 
    \[
    K_I^2 \,=\, (\k_{j+1} - \k_1)^2 \,=\, \exp\left[ -\frac{1}{2} R_{jj}\right], \quad K_{\tilde{I}}^2 \,=\, (\k_{j+1} - \k_{i+1})^2 \,=\, \exp\left[ -\frac{1}{2} (R_{jj}+ R_{ii})+ R_{ij}\right]. 
    \]
    Then the proof follows by observing that $K_J^2$ can be written as a product of squares of $2\times 2$ minors of type $I,\tilde{I}$, i.e., $K_J^2 = \prod_{1\le l_1 < l_2 \le k} K_{j_{l_1}j_{l_2}}^2$.
\end{proof}
\begin{example}
    If one uses the explicit formulas for the entries of the matrix $R$ computed in \Cref{sec:Riemannmatrix}, it is easy to verify that \eqref{eq:beta_lambda} explicitly reads
    \[
     \beta_j = \left\{ \begin{array}{ll}
\displaystyle \lambda_j^{-1} \prod_{l\not =1,j+1}^k \frac{(\k_l-\k_1)^2}{(\k_l -\k_{j+1})^2} , & \quad \text{ if } \,  1\le j\le k-1,\\
\displaystyle \lambda_j^{-1} (\k_{j+1} - \k_1)^{-2} \prod_{l=2}^k\frac{(\k_l-\k_1)^2}{(\k_l-\k_{j+1})^2}, & \quad \text{ if }\,  k\le j\le g.
    \end{array}
    \right.       
    \]

\end{example}
Therefore, using \eqref{eq:beta_lambda}, we have proven that the family of KP solitons arising from Krichever's method induce an alternative parametrization of the Hirota variety in terms of $(\underline{\lambda},\underline{\kappa})$ which is equivalent to that given in~\eqref{eq:psi_Simo}. This is summarized in the following.

\begin{theorem}\label{theo:Krichever1_Hirota}
   The variety of ${\cal H}_{\Gamma,[\kk],v_1}$ admits a parametrization
  \begin{align}\label{eq:psi_lambda}
    {\tilde \psi}_{[\kk],v_1}\, : \, \CC^{2g+1} &\dashrightarrow \KWPB\\
(\underline{\lambda},\underline{\kappa}) &\longmapsto ((\alpha_J)_{J\in {\cal B}_{\aa},v_1},(\UU,\VV,\WW)),\nonumber
\end{align}
whose image is given by
\begin{equation}\label{eq:image_lambda}
\displaystyle
    \alpha_{J} \,=\, \frac{{\tilde A}_J \, K_J}{{\tilde A}_{I_k} \, K_{I_k}} \,=\, \frac{K_J^2}{K_{I_k}^2} \prod_{l=1}^g\lambda_l^{(B^T\cc_J)_{l+1}},
\quad\quad
 \begin{bmatrix}
        \UU \,\, \VV \,\, \WW 
    \end{bmatrix} = B \, \cdot K,
\end{equation}
which is equivalent to the parametrization presented in Theorem \ref{theo:Hirota_parametrization}.
\end{theorem}

\subsection{Hirota parametrization in function of the KP divisor}\label{sec:Hirota_divisors}

From \Cref{sec:theta_functions_limit}, we know that, given the triple $(\X,p_0,\zeta)$, where $\X$ is a smooth complex algebraic curve of genus $g$ and $p_0$ is the essential singularity of the Baker--Akhiezer function, the complex smooth KP solutions are parametrized by non-special divisors  $D = p_1+\cdots+p_g-p_0$ on $\X$. 

When the family of hyperelliptic curves described in~\Cref{sec:Riemannmatrix} degenerates to $X=X_+\cup X_-$, as before, we identify the component $X_+$ with the vertex $v_1$ and $X_-$ with $v_2$. The aim of this section is to describe the Hirota variety ${\cal H}_{\Gamma,[\kk]}$ in function of the distribution of the divisor points. Since in the tropical limit, $p_0$ may end up either on $X_-$ or on $X_+$, it is convenient to work with the smooth divisors of the form $D = p_1+\dots +p_g-p_0$.

Given a divisor $D = p_1 + p_2 + \dots + p_g -p_0,\in \text{Div}(X)$, we denote by $D_+$ and $D_-$ its restrictions to $X_+,X_-$, respectively. In this section, we prove that, for any given $k\in[g]$, ${\cal H}_{\Gamma,[\kk],v_1}$ is parametrized by  divisors such that $D_+ = p_1+\dots+p_k-p_0$ and $D_- = p_{k+1}+\dots+p_g$ of degree $(g-1)$, see \Cref{the:paramHirotagenericD}. Then, by the duality argument settled in \Cref{theo:Hirota_parametrization}, it follows that ${\cal H}_{\Gamma,[\kk],v_2}$ is parametrized by the divisors such that $D_+ = p_{n-k+1}+\dots+p_g$ and $D_- = p_{1}+\dots+p_{n-k}-p_0$ of degree $(g-1)$. By \eqref{eq:Hirota_KK}, we have that 
${\cal H}_{\Gamma,[\kk]}\; = \; {\cal H}_{\Gamma,[\kk],v_1}\, \cup \, {\cal H}_{\Gamma,[\kk],v_2}$, it follows that
${\cal H}_{\Gamma,[\kk]} $ is parametrized by degree-$(g-1)$ divisors $D$ such that
\[(\deg D_+ , \deg D_- ) \, = \, (k-1,g-k) .\]

This subdivision also matches the two possible ways of distributing the divisor $D$ in function of the strongly connected orientations, since $[\kk]$ also fixes the equivalence class of strongly connected orientations of $\Gamma$ with exactly $k$ outgoing edges at $v_1$ and $n-k$ outgoing at $v_2$. Therefore, once the strongly connected orientation of the banana graph $\Gamma$ is fixed, the degrees of the divisors $D_+$ and $D_-$ are uniquely determined by the number of outgoing edges at the vertices $v_1$ and $v_2$, respectively. 

In other words, the vertices of the Voronoi polytope---equivalently, the strongly connected orientations of the banana graph---parametrize, up to equivalence, the different ways of distributing the KP divisor points of degree $(g-1)$ on the two components of the curve $X$. This observation motivates the construction of the variety of KP soliton solutions arising from a banana graph for all possible choices of such divisors and provides a new insight into ${\cal H}_\Gamma$, the \textit{Hirota variety of a banana graph}.

Note that inspecting the $\tau$-function in \eqref{eq:tau_1} is equivalent to expressing the vector ${\bf D} = (\log \beta_1,\dots,\log \beta_g)$ in function of ${\cal A} (p_1,\cdots,p_g)$, the image of the divisor $D$ via the Abel map, that is, $  {\bf D} =-{\cal A} (p_1,\cdots,p_g) - \mathbf{K}$,
where $\mathbf{K}=(\mathbf{K}_i)$ is the vector of the Riemann constants.

We now recall the construction of KP multi-solitons in  \cite{Abenda2017, nakdegeneration} associated with the reducible nodal curve $X = X_+ \cup X_-$ with the defining equation given in \eqref{eq:hyperelliptic_curve}. 
This family of curves and their KP solutions have been obtained using Krichever's approach for degenerate finite-gap solutions in  \cite{Abenda2017} and the Sato Grassmannian in
\cite{nakdegeneration}. These constructions provide a representation of KP solitons equivalent to that obtained in \Cref{sec:parametrization} using a purely tropical approach.

\begin{remark}
Assuming that the divisor $D$ represents a real or complex KP multi-soliton solution for soliton data in $\Gr(k,n)$ and essential singularity $p_0\in X_+$, gives a differential \emph{Darboux operator} of degree $k$. Its characteristic equation at $(x,y,t)=(0,0,0)$ has $k$ roots (counted with multiplicities), which provide the affine coordinates of the KP divisor $D_+$ on $X_+$ in the Krichever normalization of the KP wave function. In Krichever's approach to degenerate finite-gap solutions \cite{Abenda2017, Krichdeg}, the divisor $D_-=p_{k+1}+\cdots +p_g$ is the restriction of the KP divisor to $X_-$ precisely when it is the pole divisor of a meromorphic function on $X_-$ whose values at the nodes agree with those of the KP wave function on $X_+$ for all $(x,y,t)$.

Finally, the Darboux operator coincides with the \emph{Sato dressing operator}. This equivalence implies that the same soliton solutions can be described by points in the Sato Grassmannian. For convenience, Nakayashiki in \cite{nakimrn} constructs these points via the adjoint KP wave function; that is, they correspond to the KP soliton 
associated with the divisor $D = D_++D_-$, up to duality transformation between $\Gr(k,n)$ and $\Gr(n-k,n)$, which implements the space-time transformation $(x,y,t)\mapsto (-x,-y,-t)$.
\end{remark}

We adopt the notation used in \cite[Section~6]{AgoFevManStu}. More precisely, let $Z=X_+\cap X_- =\{q_1,q_2,\dots,q_n\}$ and consider the normalization $\nu:\PP^1\to X_+$ where $\kappa_i\coloneqq\nu^{-1}(q_i)$, for $i=1,\dots,n$. Then define
\begin{equation}\label{eq:polynomialK}
    K(z)\coloneqq \prod_{j=1}^n(z-\kappa_j),
\end{equation}
with $K'(z)$ denoting the derivative of the polynomial $K(z)$, where $z$ is the affine coordinate in $\PP^1$ such that $p_0$ is the infinite point. Let $D=p_1+\cdots+p_g-p_0 \in \text{Div}(X)$, and $D_+$, $D_-$ their restrictions, respectively, to $X_+$, $X_-$. Assume that the first $k$ points in $D$ are supported on the component $X_+$ to which $p_0$ belongs, and define the polynomial $P(z), Q_l(z)$ as in \eqref{eq:polynomialsPandQ},  where $\langle Q_1,Q_2, \dots,Q_{n-k}\rangle$ is a basis of the space $H^0(X_-,D_-)$. Therefore, we have $n - k = \deg(D_-)+1-p_a(X_1) = g-k +1$. Furthermore, as $\deg(D_--Z) = g-k - n = -k -1 < 0 $, then $H^0(X_-,D_--Z)=0$. 

The parameters $\lambda_j$ appearing in $\tilde A$, as defined in~\eqref{eq:matrix_A_JGP}, were expressed in terms of the KP divisor $D=p_1+\cdots+p_g-p_0$ in \cite{Abenda2017}. 
In that work, the primary goal was to characterize the properties of the KP divisor when the corresponding KP multi-soliton  is real and regular. The main result, \cite[Theorem~9.1]{Abenda2017}, provides such a characterization for divisors 
supported on smooth (Equation~(53)) and non-smooth (Equation~(54)) points and such that they satisfy the Dubrovin--Natanzon condition~\cite{DubNat}. 
In the case of smooth divisors, \cite[Theorem~9.1]{Abenda2017} can be adapted in a straightforward way to the more general case of the complex KP multi-solitons, as follows.

\begin{theorem}\label{theo:lambda_D}
Let $\k_1,\dots,\k_n$ be fixed. For any smooth divisor $D=p_1+\cdots+p_g-p_0$ on $X$ whose restriction to $X_+$ is $D_+=p_1+\cdots+p_k-p_0$, there exists an associated  matrix $\tilde A$ as in~\eqref{eq:matrix_A_JGP} such that the corresponding $g$-tuple $(\lambda_1,\dots,\lambda_g)$ satisfies
\begin{equation}\label{eq:lambda_div}
    \lambda_j \,=\, \frac{\displaystyle P(\kappa_1) K' (\kappa_1)\prod_{l=1}^{n-k} Q_l(\kappa_1)}{\displaystyle P(\kappa_{j+1})K' (\kappa_{j+1})\prod_{l=1}^{n-k} Q_l(\kappa_{j+1}) } , \quad\quad j\in [g],
\end{equation}
with $P(z),Q[z]$ as in~\eqref{eq:polynomialsPandQ}, and $K'(z)$ the derivative of $K(z)$ in~\eqref{eq:polynomialK}.
\end{theorem}

Let us now recall the characterization of the same family of soliton solutions in the Sato Grassmannian. 
Algorithm 27 in \cite{AgoFevManStu} provides a recipe for constructing the KP soliton corresponding to the Sato Grassmannian point $U = \iota(H^0(X,D+\infty p))$. The output is the $(n-k)\times n$ matrix 
\[{A}^\vee \, = \, \begin{bmatrix}
   \displaystyle  \frac{P(\kappa_1)}{K'(\kappa_1)} & \displaystyle\frac{P(\kappa_2)}{K'(\kappa_2)}& \dots & \displaystyle\frac{P(\kappa_n)}{K'(\kappa_n)}\\
   \displaystyle \frac{P(\kappa_1)Q_1(\kappa_1)}{K'(\kappa_1)} & \displaystyle\frac{P(\kappa_2)Q_1(\kappa_2)}{K'(\kappa_2)}& \dots & \displaystyle\frac{P(\kappa_n)Q_1(\kappa_n)}{K'(\kappa_n)}\\
    \vdots & \vdots & \dots & \vdots\\
 \displaystyle\frac{P(\kappa_1)Q_{n-k}(\kappa_1)}{K'(\kappa_1)} & \displaystyle\frac{P(\kappa_2)Q_{n-k}(\kappa_2)}{K'(\kappa_2)}& \dots &\displaystyle \frac{P(\kappa_n)Q_{n-k}(\k_n)}{K'(\kappa_n)}
\end{bmatrix}.\]
Note that this coincides with the transpose of the matrix appearing in~\cite[Theorem~7.4]{nakdegeneration} for $m_0=g$ under the following correspondence of notation: the quantities $k, \k_i, d_i, \Lambda_i, D_i$ used in \cite{nakdegeneration} correspond, respectively, to $g-k, \lambda_i, p_i, K'(\kappa_i), P(\k_i)$ in our notation. Moreover, the coefficients $C_{ij}$ arise from a different choice of basis for $H^0(X_-,D_1)$. 

Since the algorithm relies on the choice of coordinates for the Sato Grassmannian originally introduced in \cite{nakimrn}, the KP wave function associated with the KP multi-soliton represented by $A^\vee$ is the adjoint of the KP wave function corresponding to the divisor $D$.

\begin{theorem}
    Let $\k_1,\dots,\k_n$ and $D=D_+ + D_-$ be as above. Then ${A}^\vee$ is an $(n-k)\times n$ matrix representing the KP multi-soliton obtained from the one associated with the $k\times n$ matrix $A$ in \eqref{eq:matrix_A} under space--time inversion $(x,y,t) \mapsto (-x,-y,-t)$. More precisely, if $$u_{A} (x,y,t) \,=\, 2\partial^2_x \tau_{A}(x,y,t) \qquad \text{and} \qquad u_{A^\vee} (x,y,t) \,=\, 2\partial^2_x \tau_{A^\vee}(x,y,t)$$ are the KP solutions represented by the matrices $A$ and $A^\vee$, respectively, then
\[
u_{A} (x,y,t) \,=\, u_{A^\vee} (-x,-y,-t).
\]
\end{theorem}
\begin{proof}
 The proof follows from \cite[Lemma 10.1]{Abenda2017} and \Cref{theo:lambda_D}. Indeed, the matrix $A$ in~\eqref{eq:matrix_A} yields the same KP solution as $\tilde A$ in \eqref{eq:matrix_A_JGP} when the vectors  $\underline \beta$ and $\underline \lambda$ are chosen as in \eqref{eq:beta_lambda}, that is, $u_{A} (x,y,t) =u_{\tilde A}(x,y,t)$. 
\end{proof}

Comparing \eqref{eq:lambda_div} and \eqref{eq:beta_lambda}, we get the parametrization of the variety ${\cal H}_{[{\bf k}],v_1}$ in terms of the divisor $D=p_1+\dots+p_g-p_0$ for the soliton data in $\Gr(k,n)$, with $p_0,p_1,\dots,p_k\in X_+$ and $p_{k+1},\dots,p_g\in X_-$.  

\begin{theorem}\label{the:paramHirotagenericD}
Let $\underline{D} =(p_1,\dots,p_g)$, where $D=p_1+\cdots+p_g-p_0$ is the KP divisor for the soliton data in $\Gr(k,n)$, with $D_+$ and $D_-$ as above.
 Then, the irreducible variety ${\cal H}_{[{\bf k}],v_1}$ admits the parametrization
  \begin{align}\label{eq:Hirota_div}
    {\hat \psi}_{[\kk],v_1}\, : \, \CC^{2g+1} &\dashrightarrow \KWPB\\
(\underline{D},\underline{\kappa}) &\longmapsto ((\alpha_J)_{J\in {\cal M}},(\UU,\VV,\WW)),\nonumber
\end{align}
whose image is given by
\begin{equation}\label{eq:image_D}
\displaystyle
    \alpha_{J} \,=\, \frac{K_J^2}{ K_{I_k}^2} \, \prod_{l=1}^g \left[ \frac{P(\k_{l+1}) \prod_{j=1}^{n-k} Q_j(\k_{l+1})K' (\k_{l+1})}{P(\k_{1}) \prod_{j=1}^{n-k} Q_j(\k_{1})K'(\k_1)} \right]^{-(B^T\cc_J)_{l+1}}
\quad
 \begin{bmatrix}
        \UU \,\, \VV \,\, \WW 
    \end{bmatrix} = B \, \cdot K,
\end{equation}
which is equivalent to the parametrizations presented in~\Cref{theo:Hirota_parametrization} and \Cref{theo:Krichever1_Hirota}.   
\end{theorem}

We remark that ${\cal H}_{[{\bf k}],v_2}$, by the duality property settled in \Cref{theo:Hirota_parametrization}, admits a parametrization in terms of the divisor $D=p_1+\dots+p_g-p_0$ for the soliton data in $\Gr(n-k,n)$, with $p_0,p_1,\dots,p_{n-k}\in X_-$ and $p_{n-k+1},\dots,p_g\in X_+$. In this case, the image of the map ${\hat \psi}_{[\kk],v_2}\, : \, \CC^{2g+1} \dashrightarrow \KWPB$ is given by 
\[
\displaystyle
    \alpha_{\bar{J}} \,=\, \frac{K_{\bar{J}}^2}{ K_{\bar{I}_k}^2} \, \prod_{l=1}^g \left[ \frac{\bar{P}(\k_{l+1}) \prod_{j=1}^{k} \bar{Q}_j(\k_{l+1})K' (\k_{l+1})}{\bar{P}(\k_{1}) \prod_{j=1}^{n-k} \bar{Q}_j(\k_{1})K'(\k_1)} \right]^{-(B^T\cc_J)_{l+1}}  
\quad\quad
 \begin{bmatrix}
        \UU \,\, \VV \,\, \WW 
    \end{bmatrix} = -B \, \cdot K,
\]
with $\bar{J} = [n]\setminus J$, $\bar{P}(z) =\prod_{i=1}^{n-k} (z-p_i)$, $\bar{Q}_1 (z)=1$, $\bar{Q}_2 (z)=1/(z-p_{n-k+1}),\dots ,\bar{Q}_{k} (z)=1/(z-p_{g})$, and $K'(z)$ as in \eqref{eq:polynomialK}.

We end this section by remarking that \eqref{eq:image_D} contains the explicit relation between ${\bf D}$ and the image of the divisor $D$ via the Abel map.
Recall that ${\mathcal A}_j (p_1,\dots,p_g)$, the $j$-th component of the Abel map, may be expressed as in~\eqref{eq:AbelJacobi}.
Then, comparing \eqref{eq:alpha_J} and \eqref{eq:image_D}, we get
\[
\log(\beta_j)\, =\, -\mathcal{A}_j (p_1,\dots,p_g) +\log\left( \frac{K'(\k_{j+1})}{K'(\k_{1})} \right)+
\frac{k}{2} R_{jj}-\sum_{l=1}^{k-1} R_{jl}.
\]
We remark that the formula above contains the components of the tropical limit of the vector of Riemann constants $\bf K$ in \eqref{eq:riemannconstant} up to a shift.

\subsection{Positive Hirota varieties of banana graphs}\label{sec:Hirota_pos}

The KP multi-solitons form a distinguished class of solutions of the KP equation. Since this equation was originally motivated by the study of shallow water waves, its real and regular solutions are the ones which are relevant for applications. Following \cite{DubNat}, real and regular finite-gap KP solutions are parametrized by KP divisors à la Dubrovin--Natanzon on M-curves. According to \cite{KW2013}, a KP multi-soliton solution generated by the soliton data $\{ A\in \Gr(k,n), \k_1,\dots,\k_n\}$ is real and regular if and only if $\k_1<\dots<\k_n$ and $A$ is in the totally non-negative part of the Grassmannian, that is all Pl\"ucker coordinates of $A$ are real and share the same sign. The combinatorial properties of totally non-negative Grassmannians, first unveiled in Postnikov's seminal work \cite{Pos}, have been used by Abenda and Grinevich \cite{Abenda2018,Abenda2019,Abenda2022} to prove that every real regular KP multi-soliton arises as a limit of a real regular finite-gap solution. More precisely, they proved that, for any real and regular KP multi-soliton, one may associate to its soliton data a spectral curve which is an MM-curve in the sense of \cite{kummer2024maximal} and a KP divisor satisfying the reality and regularity conditions settled by Dubrovin and Natanzon in \cite{DubNat}. This result establishes a direct bridge between the positive Grassmannian description of soliton data and the algebraic–geometric description of real and regular finite-gap solutions on M-curves developed in \cite{DubNat}.

In the case of tropical curves associated with banana graphs, the condition $\k_1<\dots<\k_n$ ensures that the tropical curve is an MM-curve in the sense of \cite{kummer2024maximal}, that is, the tropical degeneration of a maximal real hyperelliptic curve of genus $g=n-1$.
Using the explicit form of the matrix $A$ in \eqref{eq:matrix_A}, the positivity of all maximal minors is equivalent to the additional requirement that $\beta_1,\dots,\beta_g>0$. In other works, the corresponding soliton solutions are associated with points in the totally positive part of the Grassmannian, which we denote $\Gr_{>0}(k,n)$. These conditions are consistent with those found in \cite{Abenda2017}, where it is proven that the KP soliton associated with the data $(\underline \lambda, \underline \k)$ is real and regular if and only if $\k_1<\dots<\k_n$ and $\lambda_1,\dots,\lambda_g>0$, that is, the matrix $\tilde A$ in \eqref{eq:matrix_A_JGP} is totally positive.

In \cite{Abenda2017}, it was also shown that the KP divisor associated with a real and regular KP solution on $X$ satisfies the Dubrovin--Natanzon conditions, i.e., each finite oval of the MM-curve contains exactly one divisor point. In the affine coordinates used throughout this manuscript, if we fix the combinatorial type $[\kk]$, and write $D=D_++D_-$, with $D_+=p_1+\cdots+p_k-p_0$, and $D_- =p_{k+1}+\cdots p_{g}$, then the positivity condition $A \in  \Gr^+(k,n)$ expressed as the Dubrovin--Natanzon condition. This implies that there exists a permutation $\pi$ of the indices $\{1,\dots,g\}$ such that
\begin{equation}\label{eq:pos_cond}
    \k_1 < p_{\pi(1)} < \k_2 < \cdots < \k_{g} < p_{\pi(g)} < \k_{g+1}.
\end{equation}
We now introduce the following definition.
\begin{definition}\label{def:positiveHirota}
   The \textit{real and positive Hirota variety}  ${\mathcal H}_{\Gamma,[\kk],v_1} (\RR)^+$ associated with the equivalence class $[\kk]$ and the graph vertex $v_1$ is the image of the restriction of the map ${\hat \psi}_{[\kk],v_1}$, defined in~\eqref{eq:Hirota_div}, 
to the pairs $(\underline{D}, \underline{\kk}) \in \RR^{2g+1}$ satisfying $D_+=p_1+\cdots+p_k -p_0$, $D_+ =p_{k+1}+\cdots p_{g}$ and \eqref{eq:pos_cond}.
\end{definition}
For a fixed combinatorial type $[\kk]$ of the divisor $D$ on the tropical curve, the real and positive Hirota variety ${\mathcal H}_{\Gamma,[\kk],v_1} (\RR)^+$ thus provides a positive-geometric counterpart of the space of real and regular finite-gap data associated with maximal real algebraic curves in the limit \cite{Abenda2018}.  
Note that ${\mathcal H}_{\Gamma,[\kk],v_2} (\RR)^+$ has a similar characterization, with the divisor $D$ splitting as $D_+= p_{n-k+1}+\cdots + p_g$, $D_-= p_1+\cdots+ p_{n-k}-p_0$. Therefore,
$${\mathcal H}_{\Gamma,[\kk]} (\RR)^+ \,=\,{\mathcal H}_{\Gamma,[\kk],v_1} (\RR)^+ \cup {\mathcal H}_{\Gamma,[\kk],v_2} (\RR)^+$$ 
is parametrized by degree-$(g-1)$ divisors $D$ satisfying~\eqref{eq:pos_cond} and such that
\[(\deg D_+ , \deg D_- ) \, = \, (k-1,g-k) .\]
Finally, the positive part of ${\cal H}_{\Gamma,\cal{V}_Q}$, denoted ${\cal H}_{\Gamma,\cal{V}_Q}(\RR)^+$, is obtained as the union of ${\mathcal H}_{\Gamma,[\kk]} (\RR)^+$ over all equivalence classes $[\kk]$ of Voronoi vertices. It is therefore the union of the positive regions arising from several positive Grassmannian points corresponding to all admissible combinatorial types $(\deg D_+ , \deg D_- ) \, = \, (k-1,g-k)$, $k\in [g]$ that satisfy \eqref{eq:pos_cond}.
This combinatorial description suggests that the closure of ${\cal H}_{\Gamma,\cal{V}_Q}(\RR)^+$ is connected and coincides with the positive part of ${\cal H}_{\Gamma}$; in other words, it provides a faithful characterization of the positive part of the tropical theta divisor in terms of KP divisors à la Dubrovin--Natanzon~\cite{DubNat}. We will explore this question further in a subsequent work.

\section{Hirota varieties of Delaunay polytopes}\label{sec:Hirota_of_hypers}

In this section, we study the polynomial equations that describe the Hirota variety associated to Delaunay polytopes that are combinatorially equivalent to hypersimplices.

The Hirota variety ${\cal H}_{\cal D}$ of a Delaunay set ${\cal D}$ 
containing $d$ points was first introduced in \cite{AgoFevManStu}, as the variety of all points $((\alpha_\cc)_{\cc\in {\cal D}},(\UU,\VV,\WW))$ in the parameter space $(\CC^*)^d\times \mathbb{W}\PP^{3g-1}$ such that the expression $\tau_{\cal D}(x,y,t)$ as in~\eqref{eq:tau_from_theta} gives a KP solution. 

One of the motivations for distinguishing between Hirota variety of a graph and that of a Delaunay polytope is that different tropical curves may give rise to the same Delaunay set; see, for instance, the list of genus~$3$ tropical curves and their Delaunay sets in~\cite[Section~4]{AgoCelStrStu}. Consequently, a given $\tau$-function $\tau_{\cal D}(x,y,t)$ may be realized as a KP soliton associated with several different curves. The Hirota variety of a graph eliminates this ambiguity by fixing all the algebro-geometric data $(X,p_0,\zeta,D)$ in the tropical limit. 

For example, the Hirota variety ${\cal H}_{\cal D}$ of a $3$-dimensional simplex parametrizes solutions arising from curves given by the union of four lines, the union of two conics, or a conic together with two lines. The component of ${\cal H}_{\cal D}$ parametrizing solutions arising from a fixed curve is determined by the data of the Riemann matrix of the underlying tropical curve, the matroid, and the periods of meromorphic differentials in the tropical limit, all of which appear explicitly in the parametrization of $((\alpha_\cc)_{\cc\in {\cal D}},(\UU,\VV,\WW))$. 

This distinction was unnecessary in \cite{FevMan}, since the hypercube occurs as Delaunay polytope only for the family of irreducible rational nodal curves.

The equations defining the Hirota variety associated to a Delaunay polytope were first studied in \cite[Corollary~6]{AgoFevManStu}. They can be described by looking at the points in the set 
\[{\cal D}^{[2]} \, \coloneqq \, \{ \cc_i+\cc_j \,\, | \,\, \cc_i,\cc_j\in {\cal D}, \cc_i\neq \cc_j\}.\]
We say that a point ${\bf d}\in {\cal D}^{[2]}$ is \textit{uniquely attained}
if there exists precisely one index pair $(i,j)$ such that $\cc_i+\cc_j={\bf d}$. In that case, $(i,j)$ is a \textit{unique pair} and contributes the quartic 
\begin{equation}\label{eq:quartic_uniquely}
    P_{ij}(\UU,\VV,\WW)\,\coloneqq\, P((\cc_i-\cc_j)\cdot \UU, (\cc_i-\cc_j)\cdot \VV, (\cc_i-\cc_j)\cdot \WW),
\end{equation}
where $P(x,y,t)=x^4-4xt+3y^2$ is the KP dispersion relation. Note that the expression in~\eqref{eq:quartic_uniquely} does not change if we switch $i$ and $j$. When ${\bf d}\in {\cal D}^{[2]}$ is not uniquely attained, we associate it with the polynomial
\begin{equation}
   P_{{\bf d}} (\UU,\VV,\WW)\, \coloneqq\, \sum_{\cc_i+\cc_j={\bf d}} P_{ij} (\UU,\VV,\WW) \alpha_{\cc_i}\alpha_{\cc_j}.
\end{equation}
We now study these equations in detail in the case where ${\cal D}$ is combinatorially equivalent to a hypersimplex $\Delta_{k,n}$.

\subsection{Hirota varieties of hypersimplices}

Let $\Gamma$ be a genus $g$ metric graph with tropical Riemann matrix $Q$ as introduced in \Cref{sec2}. For simplicity, we denote by ${\cal D} = {\cal D}_{\aa,Q}$ the Delaunay set of $\Gamma$ for the fixed Voronoi vertex $\aa$ and write $d$ for its cardinality. We assume that ${\cal D}$ is combinatorially equivalent to the hypersimplex $\Delta_{k,n}$. As motivated in \cite{FevMan} in the case of irreducible rational nodal curves, we are interested in the subvarieties $H_{ \Delta_{k,n}}^I$ of the Hirota varieties ${\cal H}_{\Delta_{k,n}}$ given by the Zariski closure of the set of points $((\alpha_\cc)_{\cc},(\UU,\VV,\WW))\in {\cal H}_{\Delta_{k,n}}$ such that $\UU\neq {\bf 0}$.

As from \Cref{cor:Delaunay_in_Rg} the polytope $\mathcal{D}$ is combinatorially equivalent to the hypersimplex $\Delta_{k,n}$, we can compute the elements in $\mathcal{D}^{[2]}$ and their multiplicities by looking at the set $\Delta_{k,n}^{[2]}$. In what follows, we describe this set combinatorially. 
\begin{lemma}
    We have 
   $$\Delta_{k,n}^{[2]} \,=\, \left\{ \sum_{i \in I} \ee_i + 2\sum_{j \in J} \ee_j\,:\, I \in \binom{[n]}{2k-2d}, \, J \in \binom{[n]}{d}, \, 0 \leq d \leq k-1, \, I\cap J=\emptyset 
   \right\} \,.$$

   Fixing $d$ to be the number of $2$ coordinates, each point $\sum_{i \in I} \ee_i + 2\sum_{j \in J} \ee_j$ in $\Delta_{k,n}^{[2]}$ is attained $\frac{1}{2}\binom{2k-2d}{k-d}$ times, and there are $\binom{n}{d}\binom{g+1-d}{2k-2d}$ such points.
\end{lemma}

\begin{proof}
   It is clear that any sum of two points in $\Delta_{k,n}$ will have coordinates $0, 1,$ or $2$. Furthermore, the total sum of the coordinates must be exactly $2k$. Letting $d$ be the number of coordinates equal to $2$, we find that there are $\binom{n}{d} \cdot \binom{n-d}{2k-2d}$ points satisfying the above conditions. Furthermore, each such point can be expressed as a sum of two points in $\Delta_{k,n}$ in exactly $\frac{1}{2}\binom{2k-2d}{k-d}$ ways, as such an expression involves distributing the $2k-2d$ ones among the two summands.
   
   \end{proof}

Note that $\Delta_{1,n}^{[2]}$ coincides with the set of vertices of $\Delta_{2,n}$. Furthermore, to obtain the points in the original set $\DaQ^{[2]}$ it is enough to shift the points in $\Delta_{k,n}^{[2]}$ by the vector $-2\mathbf{s}_{\aa}$. 

We remark that the maximum number of times a point in $\Delta_{k,n}^{[2]}$ is attained rises quickly with $k$, and is equal to $\frac{1}{2}\binom{2k}{k}$. The first few values are $1, 3, 10, 35, 126$. This leads to increasingly complicated expressions for generators of the Hirota variety. Since $k$ is bounded by $\frac{n}{2}$, the complexity of generators of Hirota varieties increases with the genus. The number of equations cutting out the Hirota variety also rises with the genus, however we may simplify the list using the following result, which shows that each equation corresponds to a face of $\Delta_{k,n}$ and parallel faces contribute redundant equations. We define the \emph{direction} of a face $F$ to be the set of coordinates in which points of $F$ may vary, i.e. $\{i: x_i \text{ is not fixed on } F\}$.

\begin{theorem}
    Each point of $\Delta_{k, n}^{[2]}$ corresponds to a centrally symmetric odd-dimensional face of $\Delta_{k, n}$. There are $ \binom{g+1}{2\ell}$ face directions for each dimension $1 \leq 2\ell-1 \leq g$, and all faces with the same direction contribute the same quartic as a generator of the Hirota variety.
\end{theorem}
\begin{proof}
    Throughout this proof, we use $\Delta = \Delta_{k, n}$. A point ${\bf d} = \cc_1 + \cc_2 \in \Delta$ can be associated (non-uniquely) to the line segment connecting $\cc_1$ to $\cc_2$. This determines a diagonal of some face of $\Delta$. Since $\cc_1,\cc_2\in\{0,1\}^n$ with each having exactly $k$ ones, their sum ${\bf d}$ has entries in $\{0,1,2\}$, with all entries of $\cc$ summing to $2k$. In particular, the number of ones must be even. Let $I_1 = \{\,i : (\cc_1)_i=1\}$, and let $2\ell = |I_1|$. Then, for $i \not \in I_1$, the coordinates of $(\cc_1)_i, (\cc_2)_i$ are equal and uniquely determined by the value of ${\bf d}_i$. On $I_1$, exactly one of $\cc_1,\cc_2$ has a $1$. Hence, the minimal face $F\subset\Delta$ containing $\cc_1,\cc_2$ is
\[
F \;=\;
\bigl\{\,x\in\Delta : x_i=1\ (i: {\bf d}_i = 2 ),\;x_i=0\ (i: {\bf d}_i = 0)\bigr\}
\;\cong\;
\Delta_{\ell,\;2\ell},
\]
which has dimension $2\ell - 1.$ Since each such face $F$ of $\Delta$ is centrally symmetric, all its diagonals (if any exist) intersect at the unique center of symmetry, so any two endpoints of a diagonal of that face sum to twice that center and hence to the same vector. In this way, every point of $\Delta^{[2]}$ corresponds uniquely to a centrally-symmetric odd-dimensional face of $\Delta$. The sum ${\bf d}$ is attained by the sums of endpoints of the $\frac{1}{2} \binom{2\ell}{\ell}$ diagonals of the face. Next, two faces $F,F'$ have the same direction precisely when they fix the same moving coordinate set $I_1$.  Given $I_1$, there are $\binom{g+1}{2\ell}$ such choices.  

Finally, we show that the quartic defining each Hirota generator depends only on the moving coordinates $I_1$, hence is identical (up to multiplication by a constant) for all faces in the same parallel class. 

Fix a direction $I_1$ and let $F,F'$ be two faces with that direction.  
    Let ${\bf d}=\cc_1+\cc_2$ and ${\bf d}'=\cc_1'+\cc_2'$ be the corresponding points in $\Delta^{[2]}$.
    On the fixed coordinates, the two faces differ by a vector 
    \[
        \uu \,:=\, \mathbf{1}_{F'}-\mathbf{1}_{F}\in\{0,1\}^n,
    \]
    supported outside $I_1$, so that ${\bf d}'={\bf d}+2\uu$.
    For each pair $(\cc_i,\cc_j)$ with $\cc_i+\cc_j={\bf d}$ we define a corresponding pair
    $(\tilde\cc_i,\tilde\cc_j)=(\cc_i+\uu,\cc_j+\uu)$ satisfying $\tilde\cc_i+\tilde\cc_j={\bf d}'$.
    Since $\tilde\cc_i-\tilde\cc_j=\cc_i-\cc_j$, the quartic
    \(
      P_{{\bf d}}=P_{{\bf d}'}
    \)
    is unchanged.

    The exponential coordinates $a_{\cc}$ are given by $a_{\cc}=\exp\!\bigl[\tfrac12\cc^{T} R\cc\bigr]$ for a $n\times n$ symmetric matrix $R$.
    Then
    \[
      a_{\tilde\cc_i}a_{\tilde\cc_j}
      \,=\,\exp\!\Bigl[\tfrac12(\cc_i+\uu)^T R(\cc_i+\uu) + \tfrac12(\cc_j+\uu)^T R(\cc_j+\uu)\Bigr]
      \,=\,a_{\cc_i}a_{\cc_j}\cdot
        \exp\!\bigl[\uu^T R{\bf d}+\uu^T R\uu\bigr].
    \]
    The multiplier $\exp[\uu^T R{\bf d}+\uu^T R\uu]$ depends only on $\uu$ and the common direction $I_1$, not on the particular pair $(\cc_i,\cc_j)$.  
    Therefore each term $a_{\cc_i}a_{\cc_j}$ appearing in the quartic generator for $F$ rescales by the same constant factor when passing to $F'$.
    Hence the corresponding quartics differ by an overall scalar multiple:
    \[
      \sum_{\tilde{\cc}_i+\tilde{\cc}_j={\bf d}'} P_{{\bf d}'}\,a_{\tilde\cc_i}a_{\tilde\cc_j}
      = 
      \exp[\uu^T R{\bf d}+\uu^T R\uu]
      \sum_{\cc_i+\cc_j={\bf d}} P_{{\bf d}}\,a_{\cc_i}a_{\cc_j}.
    \]
    Thus all faces with the same direction $I_1$ contribute the same quartic, up to a nonzero scalar multiple, to the ideal defining the Hirota variety. This completes the proof.
\end{proof}
 
\begin{example}[Genus-$3$]\label{ex:genus3_Hirotaeq}
Let $\aa = (-\frac{1}{2},\frac{1}{2},\frac{1}{2})\in [\mathbf{2}]\subset\mathcal{V}_Q$. From \Cref{cor:Delaunay_in_Rg}, the corresponding Delaunay polytope $\DaQ$ is combinatorially equivalent to the hypersimplex. In particular, we have $B^T(D_{\aa,Q}) + {\bf s}_\aa= \Delta_{2,4} $, where ${\bf s}_\aa = (0,0,1,1)$ and the hypersimplex $\Delta_{2, 4} = \text{Conv} (\ee_i+\ee_j \, : \, \{i,j\}\in\binom{[4]}{2}).$
The set $\Delta_{2,4}^{[2]}$ consists of all $\binom{4}{1}\cdot\binom{3}{2} = 12$ permutations of the entries of $(1, 1, 2, 0)$, each attained once, and the point $(1, 1, 1, 1)$, attained 3 times as $(1, 1, 1, 1) = (1, 1, 0, 0) + (0, 0, 1, 1) = (1, 0, 1, 0) + (0, 1, 0, 1) = (1, 0, 0, 1) + (0, 1, 1, 0).$ The face directions corresponding to edges are given by vectors $\pm(\ee_i-\ee_j)$ for $\{i,j\}\in \binom{[4]}{2}$, and contribute the quartics $ P(U_i-U_j,V_i-V_j,W_i-W_j)$. There is a unique 3-dimensional face direction contributing the quartic
\begin{align*}
&\alpha_{1100}\alpha_{0011}P(U_1+U_2-U_3-U_4, V_1+V_2-V_3-
  V_4,W_1+W_2-W_3-W_4)+\\
&\alpha_{1010}\alpha_{0101}P(U_1-U_2+U_3-U_4, V_1-V_2+V_3-
  V_4,W_1-W_2+W_3-W_4)+\\
&\alpha_{1001}\alpha_{0110}P(U_1-U_2-U_3+U_4, V_1-V_2-V_3+
  V_4,W_1-W_2-W_3+W_4),
\end{align*}
where we are identifying $\alpha_\cc = \alpha_{B^T\cc+\ssa}$.
\end{example}

We close with a brief perspective towards the tropical Schottky problem: it asks when a principally polarized tropical abelian variety $(\mathbb R^g/\mathbb Z^g,Q)$, with $Q$ a real positive semidefinite symmetric matrix, arises as the tropical Jacobian of a metric graph. For the banana graph, the tropical Riemann matrix $Q$ gives a Voronoi polytope $V_Q\subset\mathbb R^g$ with the explicit vertex set as in~\Cref{thm:vertices_voronoi_Rg}. Each equivalence class $[\mathbf k]\subset\mathcal V_Q$ corresponds to a strongly connected orientation with exactly $k$ edges leaving $v_1$ and $n-k$ leaving $v_2$ (equivalently, $n-k$ entering $v_1$); consequently, for every $\mathbf a\in[\mathbf k]$ the associated Delaunay polytope is combinatorially equivalent to the hypersimplex $\Delta_{k,n}$, as proven in~\Cref{cor:Delaunay_in_Rg}, and the edges whose orientation flips within $[\mathbf k]$ form circuits, see~\Cref{cor:circuits}. Fixing $\mathbf a\in[\mathbf k]$ determines the combinatorics of the tropical theta function and, for a marked vertex $v\in\{v_1,v_2\}$, the oriented matroid $\mathcal M_{\mathbf a,v}$. The matroids are dual to each other. The KP parameters then decompose $(\alpha_J)_{J\in\mathcal M_{\mathbf a,v}}$ and $[\mathbf U\ \mathbf V\ \mathbf W]=\pm\,B\cdot K,$
where the sign is $+$ for $v=v_1$ and $-$ for $v=v_2$ (space–time inversion). Under admissible hyperelliptic degenerations, the class $[\mathbf k]$ records how the divisor $D$ in $g$ points split between the two components $X_\pm$ of the central fibre: $k$ on $X_+$ and $g-k$ on $X_-$ if $X_+$ is the component carrying the essential singularity $p_0$; $k-1$ on $X_+$ and $g-k+1$ on $X_-$ in the other case. Hence, the Voronoi vertices encode the combinatorics governing the limiting KP divisors. Different choices of $p_0$,  correspond to different choices of divisor $D$, within the same degree partition of $D-p_0$. In summary, for banana graphs of any genus $g$ the combinatorial description of the  Voronoi polytope $V_Q$ given in~\Cref{thm:vertices_voronoi_Rg} (with vertices grouped into the classes $[\mathbf k]$) gives a way to certify that $Q$ is the Jacobian of a banana graph and indexes all KP multi-soliton families supported on that Jacobian via strongly connected orientations. This frames a \emph{KP realization problem} of the tropical Jacobian and a motivation for generalizations beyond banana graphs. In this sense, our construction is a tropical shadow of Shiota’s KP–theoretic solution of the classical Schottky problem \cite{Shiota1986Schottky}. We expect this question to reveal interesting connections also with work by Baker and Norine on 
other ways to characterize realizable divisors on tropical curves such as via chip--firing games, see \cite[Theorem~1.9]{baker2007riemann}.

\vspace{1cm}
\noindent {\bf Acknowledgments.} We would like to thank the Max Planck Institute of Molecular Cell Biology and Genetics and the Center for Systems Biology Dresden for hosting the authors and providing the opportunity to work together. SA was supported by HORIZON-MSCA-2022-SE-01-01 CaLIGOLA, COST Action CaLISTA CA21109, GNFM-INdAM, and INFN projects MMNLP and GAST. CF has received funding from the
European Union’s Horizon 2020 research and innovation programme under the Marie Sk\l odowska-Curie grant agreement No 101034255); and by the European Research Council (ERC) under the European Union’s Horizon Europe research and innovation programme, grant agreement 101040794 (10000DIGITS). YM has received support from the National Science Foundation under Award DMS2402069, the UC President's Postdoctoral Fellowship Program, and the Bob Moses Fund of the Institute for Advanced Study.

{\small
\bibliography{References}}
\bibliographystyle{abbrv}

\noindent \footnotesize{Simonetta Abenda, Università di Bologna and INFN, Sezione di Bologna
\hfill  {\tt simonetta.abenda@unibo.it}

\noindent T\"urk\"u \"Ozl\"um \c{C}elik, MPI of Molecular Cell Biology and Genetics \hfill  {\tt celik@mpi-cbg.de} \\ \& Center for Systems Biology Dresden

\noindent Claudia Fevola, INRIA Saclay, Université Paris-Saclay
\hfill  {\tt claudia.fevola@inria.fr}

\noindent Yelena Mandelshtam, University of Michigan
\hfill  {\tt yelenam@umich.edu}}

\end{document}